\documentclass[11pt,english,leqno]{amsart}
\usepackage[latin1]{inputenc}

\usepackage{csquotes}

\makeatletter
\def\@tocline#1#2#3#4#5#6#7{\relax
  \ifnum #1>\c@tocdepth % then omit
  \else
    \par \addpenalty\@secpenalty\addvspace{#2}%
    \begingroup \hyphenpenalty\@M
    \@ifempty{#4}{%
      \@tempdima\csname r@tocindent\number#1\endcsname\relax
    }{%
      \@tempdima#4\relax
    }%
    \parindent\z@ \leftskip#3\relax \advance\leftskip\@tempdima\relax
    \rightskip\@pnumwidth plus4em \parfillskip-\@pnumwidth
    #5\leavevmode\hskip-\@tempdima
      \ifcase #1
       \or\or \hskip 1em \or \hskip 2em \else \hskip 3em \fi%
      #6\nobreak\relax
    \dotfill\hbox to\@pnumwidth{\@tocpagenum{#7}}\par
    \nobreak
    \endgroup
  \fi}
\makeatother

\usepackage{hyperref}
\usepackage{endnotes}
\usepackage{mathcmd}
\usepackage{amssymb,amsmath, amscd,upgreek, amsthm}
\usepackage{mathrsfs}
\usepackage{euscript}
\usepackage[usenames,dvipsnames]{color}

\usepackage{color}
\usepackage[all,color,matrix,arrow,curve]{xypic}

\numberwithin{equation}{section}
\theoremstyle{plain}

\newtheorem{theorem}{Theorem}[section]
\newtheorem{fact}{Fact}
\newtheorem*{factn}{Fact}
\newtheorem{cor}[theorem]{Corollary}
\theoremstyle{plain}
\newtheorem{pro}[theorem]{Proposition}
\newtheorem{lemma}[theorem]{Lemma}
\theoremstyle{remark}
\newtheorem{rem}[theorem]{Remark}

\makeatletter
\@addtoreset{fact}{section}
\makeatother

\def\green#1{\textcolor{Green}{#1}}

\def\PP{{\mathfrak  P}}

\def\A{{\mathcal A}}

\def\C{{\mathbb  C}}
\def\Corps{\mathfrak F}

\def\F{{F}}
\def\G{{\rm G}}
\def\I{{{\tilde{\rm I}}}}

\def\KK{{\rm  K}}
\def\K{\KK}

\def\M{{\rm M}}

\def\N{{\mathbb N}}

\def\P{{\EuScript P}}

\def\R{{\rm R}}
\def\SS{{\rm S}}
\def\T{{\rm T}}

\def\V{{\rm V}}
\def\W{{\rm W}}
\def\X{{\rm X}}

\def\Z{{\mathbb Z}}

\def\Aa{\mathscr{A}}
\def\Bb{\EuScript{B}}
\def\BB{\mathscr{B}}

\def\Cute{\mathscr C}

\def\Dd{\EuScript{D}}

\def\Hh{{\tilde {\rm H}}}

\def\Oo{\mathfrak{O}}

\def\Rr{\mathscr{R}}

\def\Uu{\EuScript{U}}

\usepackage{accents}

\def\Wfg{\underline{\Wf}}
\def\Wg{{\underline\W}}

\def\Root{{\underline\Phi}}
\def\Coroot{\check{{\underline \Phi}}}
\def\Pig{{ \underline\Pi}}
\def\Sg{{\underline S}}

\def\Htg{{\Hh_\Z}}
\def\Ddg{{\underline{\Dd}}}
\def\ellg{\underline\ell}

\def\tg{\t}
\def\b{\beta}

\def\h{\varphi}

\def\m{{\mathfrak m}}

\def\t{\boldsymbol{\tau}}

\def\lp{\langle}
\def\rp{\rangle}
\def\>{\geqslant}
\def\<{\leqslant}

\def\Hom{{\rm Hom}}
\def\EInd{{\rm End}}

\def\GL{{\rm GL}}

\def\Ker{{\rm Ker}}
\def\Im{{\rm Im}}

\def\ind{{\rm ind}}

\def\1{{\bf 1}}

\def\XX{\tilde{\mathbf X}}
\def\Mod{{\rm Mod}(\H)}

\def\Hf{{\mathfrak H}}
\def\H{ {\tilde{\mathfrak H}}}

\def\Wf{\mathfrak W}
\newcommand\cal{\mathcal}

\def\fq{{\mathbb F}_q}

\def\val{\operatorname{\emph{val}}}
\def\c1{\mathbf C _{\bf 1}}

\def\h1{{\Hf_{\bf 1}}}

\def\Weight{{{\rm X}_*}}

\def\coweight{\lambda}

\def\w{{w_0}}
\def\wl{{}^\w\!\coweight}

\def\coroot{{\check\alpha}}

\def\root{{\alpha}}
\def\broot{{\beta}}

\def\beq{\begin{equation}}
\def\eeq{\end{equation}}
\def\barr{\begin{eqnarray}}
\def\earr{\end{eqnarray}}
\def\de{\begin{equation*}}
\def\fe{\end{equation*}}
\def\dt{\begin{eqnarray*}}
\def\ft{\end{eqnarray*}}

\def\Mod{{\rm Mod}}

\def\A{\cal A}

\def\k{k}

\def\Iw{{\rm  I}}

\def\Tp{{\rm T}}

\def\Gp{ { \rm G}}
\def\pr{{\rm  pr}}

\usepackage{comment}

 \theoremstyle{remark}%
\usepackage{hyperref}
\definecolor{webblue}{rgb}{0, 0.7, 0.5}
\definecolor{webred}{rgb}{0.2, 0.3, 0.6} 
\hypersetup{colorlinks,  citecolor=webblue,  linkcolor=webred, urlcolor=black}

\title{Mod $p$ spherical and Iwahori Hecke algebras for $p$-adic ${\rm GL}_n$}
\title{An inverse  Satake isomorphism in characteristic $p$}
\date{\today}
\author{Rachel Ollivier}

\address{Laboratoire de Mathématiques de Versailles, 
 45 avenue des États-Unis 
78035 Versailles cedex, France}

\address{Columbia University, Mathematics, 2990 Broadway, New York, NY 10027, USA}

\email{ollivier@math.columbia.edu}

\date{July 10, 2012}

\subjclass[2010]{20C08, 22E50 }
\thanks{Partially supported by  NSF grant DMS-1201376 and  project TheHoPad ANR-2011-BS01-005}

\begin{document}

\maketitle

\begin{abstract}  Let $\Corps$ be a   local field with finite residue field of characteristic $p$ and $k$ an algebraic closure of the residue field.   Let $\Gp $ be the group of $\mathfrak{F}$-points of a $\Corps$-split connected reductive group.
In  the  apartment corresponding to a   maximal split torus of $\Tp$,  we choose
%a chamber and 
a hyperspecial  vertex and denote by
$\K$ the corresponding maximal compact subgroup of $\Gp$.
  Given an   irreducible smooth  $k$-representation $\rho$ of $\K$, we construct an isomorphism from  the affine semigroup algebra $k[\X^+_*(\Tp)]$ of the dominant cocharacters of $\T$ onto the Hecke algebra $\cal H(\Gp, \rho)$. In the case when the derived subgroup of $\Gp$ is simply connected, we prove  furthermore that our isomorphism is the inverse to the Satake isomorphism constructed by Herzig in \cite{satake}.

%vertex of this chamber. Let $\I$ and $\K$  be  respectively the corresponding 
% pro-$p$ Iwahori  subgroup and  the corresponding  maximal compact subgroup of $\Gp$.

Our method consists in attaching  to $\rho$  a  commutative subalgebra $\EuScript B_\rho$ of the pro-$p$ Iwahori-Hecke $k$-algebra of $\Gp$  that is isomorphic   to $k[\X^+_*(\Tp)]$.
Using a theorem by  Cabanes  which relates categories of $k$-representations of  parahoric subgroups of $\Gp$ and   Hecke modules (\cite{Cabanes}),
we then construct a natural
isomorphism from    $\EuScript B_\rho$ onto $\cal H(\Gp, \rho)$.

\end{abstract}

%\begin{altabstract}

%\end{altabstract}

\setcounter{tocdepth}{2} 

\tableofcontents

\section{Introduction}

\subsection{Framework\label{lasko}} Let  $\Corps$ be a nonarchimedean locally compact field  with ring of integers $\Oo$, maximal
ideal $\PP$ and residue field  $\mathbb F_q$ where $q$  is a power of  a prime number $p$. We fix a uniformizer $\varpi$ of $\Oo$ and choose the valuation $\val_{\mathfrak{F}}$ on   $\mathfrak{F}$  normalized by $\val_{\mathfrak{F}}(\varpi)=1$. 
%  We choose a uniformizer
%$\uppi$ and fix the valuation $\val $ normalized by $\val (\uppi)=1$.
%We consider the group $\Gp={\rm GL}_3(\Corps)$  and its smooth representations with c\oe fficients in an algebraically closed field with caracteristic $p$. We will denote by
Let $\Gp := \mathbf{G}(\mathfrak{F})$ be the group of $\mathfrak{F}$-rational points of a connected reductive group $\mathbf{G}$ over $\mathfrak{F}$ which we assume to be $\mathfrak{F}$-split. 
We fix an  algebraic closure $k$ of $\mathbb F_q$ which is the field of coefficients of (most of) the representations we  consider. All representations of $\Gp$ and its subgroups are  smooth.

Let $\mathscr{X}$ (resp.   $\mathscr X^{ext}$) be the semisimple  (resp. extended)  building of $\Gp$ and  $\rm pr: \mathscr X^{ext}\rightarrow \mathscr X$  the canonical projection map.
We fix a maximal $\Corps$-split torus $\T$ in $\Gp$ which
is equivalent to choosing an apartment $\Aa$ in $\mathscr{X}$  (see \ref{rootG}). We fix
a chamber $C$ in $\Aa$  as well as a hyperspecial vertex $x_0$ of $C$. The stabilizer of $x_0$ in $\Gp$ contains a good maximal compact subgroup $\K$ of $\Gp$ which in turns contains an Iwahori subgroup $\Iw$ that fixes $C$ pointwise. Let $\mathbf{G}_{x_0}$ and $\mathbf{G}_C$ denote the Bruhat-Tits group schemes over $\mathfrak{O}$ whose $\mathfrak{O}$-valued points are $\K$ and $\Iw$ respectively. Their reductions over the residue field $\mathbb{F}_q$ are denoted by $\overline{\mathbf{G}}_{x_0}$ and $\overline{\mathbf{G}}_C$. By \cite[3.4.2, 3.7 and 3.8]{Tit}, %$\overline{\mathbf{G}}_{x_0}$ is connected reductive and $\mathbb{F}_q$-split and  $\overline{\mathbf{G}}_C$ is connected. Hence 
$\overline{\mathbf{G}}_{x_0}$ is connected  reductive and $\mathbb F_q$-split.
Therefore we have ${\mathbf{G}}_C^\circ(\Oo)={\mathbf{G}}_C(\Oo) = \Iw$ and ${\mathbf{G}}_{x_0}^\circ(\Oo)={\mathbf{G}}_{x_0}(\Oo) = \K$.  
%In fact, $\mathbf G$ extends to a smooth $\Oo$-group
%scheme \cite[3.8.1]{Tit} (see also \cite[?]{satake}) isomorphic to ${\mathbf{G}}_{x_0}^\circ$. \mag{Lemma 3.3 de Herzig. Voir aussi ST quelque part}
\medskip

Denote  by $\overline{\mathbf{B}}$ the Borel subgroup of $\overline{\mathbf{G}}_{x_0}$ image of 
the natural morphism $\overline{\mathbf{G}}_C \longrightarrow \overline{\mathbf{G}}_{x_0}$ and by 
 $\overline{\mathbf{N}}$  the unipotent radical of $\overline{\mathbf{B}}$ and $\overline{\mathbf{T}}$ its Levi subgroup. Set
\begin{equation*}
    \K_1 := \Ker \big(\mathbf{G}_{x_0}(\mathfrak O) \xrightarrow{\; proj \;} \overline{\mathbf{G}}_{x_0} (\mathbb{F}_q) \big) \quad\textrm{and}\quad \I := \{g \in \K : proj(g) \in \overline{\mathbf{N}}(\mathbb{F}_q) \}.
\end{equation*}
Then we have a chain  $  \K_1 \subseteq \I \subseteq \Iw \subseteq \K$
of compact open subgroups in $\Gp$ such that
\begin{equation*}\K/\K_1 =  \overline{\mathbf{G}}_{x_0} (\mathbb{F}_q) \supseteq \Iw/\K_1 = \overline{\mathbf{B}}(\mathbb{F}_q) \supseteq \I/\K_1 = \overline{\mathbf{N}}(\mathbb{F}_q) \ .
\end{equation*}
The subgroup $\I$ is pro-$p$ and is called the pro-$p$ Iwahori subgroup. It is a maximal  pro-$p$ subgroup in $\K$. The quotient $\Iw/\I$ identifies with $ \overline{\mathbf{T}}(\mathbb{F}_q)$. 
%We set $\mathbb B= \overline{\mathbf{B}}(\mathbb{F}_q)$, 
%$\mathbb N= \overline{\mathbf{N}}(\mathbb{F}_q)$, $\mathbb T= \overline{\mathbf{T}}(\mathbb{F}_q)$ and $\Gf=\mathbf{G}_{x_0}(\fq)$.

%Let $n\geq 2$ and 
%denote by $\Gp$ the group of  $\Corps$-valued  points of the  general linear group ${\rm GL}_n$, by $Z$ its center, by $\K$ the maximal compact ${\rm GL}_n(\Oo)$,  
%by  $\Iw$ the standard upper Iwahori subgroup of $\K$ and $\Iw_1$ the unique pro-$p$-Sylow of $\Iw$.
%Each of these parahoric subgroups of $\K$ contains the first congruent subgroup $\K_1=1+\mathcal{M}_n(\PP)$.
%The quotient
%$\K/\K_1$ identifies with the finite group ${\rm GL}_n(k)$. Denote by $\B$ the upper Borel subgroup
%of ${\rm GL}_n(k)$ with Levi decomposition $\B=\T\U$. The quotients $\Iw/\K_1$ and $\I/\K_1$ identify respectively with $\B$ and $\U$.

% and by $\Iw(1)$ its unique pro-$p$-Sylow. It contains the first congruent subgroup $\K_1$ of the matrices in $ \K$ which are congruent to the identity modulo $\uppi$. .

\medskip

\medskip

Let $\XX  := \ind_\I^\Gp(1) $ denote the compact induction 
of the trivial character of $\I$ (with values in $k$). We see it as the space of $k$-valued functions with compact support in $\I\backslash \Gp
$, endowed with the action of $\Gp$ by right translation. The  Hecke  algebra of the $\Gp$-equivariant endormorphisms of $\XX$ will be denoted by $\Hh$. It is a $\k$-algebra. 
%We often will identify $\Hh$, as a left $\Hh$-module, via the map
%\begin{align*}
  %  \Hh & \xrightarrow{\; \cong \;} \ind_\I^\Gp(1)^\I \subseteq \XX \\
 %   h & \longmapsto (\1_\I) h
%\end{align*}
%(where $\1_\I \in \XX$ denotes the characteristic function of $\I$) with the submodule $\XX^\I$ of $\I$-fixed vectors in $\XX$.  
%Recall that, for any smooth $\Gp$-representation $V$, the subspace $V^\I = \Hom_{k[\Gp]}(\ind_\I^\Gp(1),V)$ of $\I$-fixed vectors in $V$ naturally is a right $\Hh$-module  and  $V^\I$ is non zero if  $V$ is  non zero. We denote by $\Ff$ the functor
%\begin{equation}\label{Foncteur}\Ff: V \rightarrow V^\I\end{equation} taking a smooth representation $V$ onto the $\Hh$-module $V^\I$.
 \medskip
 
 \begin{rem}Throughout the article, we will use accented letters such as  $\XX$, $\Hh$, $\H$, $\tilde \W$, $\tilde \X_*(\T)$ even when their non accented versions do not necessarily come into play: in doing so, we want to emphasize the fact that we work with the pro-$p$ Iwahori subgroup $\I$ and the  attached objects. The non accented letters are kept for the classical root data, universal representations, affine Hecke algebra \emph{etc.}  attached to the chosen Iwahori subgroup $\Iw$. \label{rema:tilda}
\end{rem}

\medskip

The algebra $\Hh$ is relatively  well understood: an integral Bernstein basis has been described by Vign\'eras (\cite{Ann})
who underlines the existence of a commutative subalgebra denoted by $\cal A^{+, (1)}$ in $\Hh$ that contains the center  of $\Hh$ and such that
$\Hh$ is finitely generated over $\cal A^{+, (1)}$. 
%Note also that $\Hh$ contains a complete family of  central orthogonal idempotents $\e_\gamma$ indexed by the orbits $\gamma\in \hat \T/\mathfrak S_n$ of the action of the symmetric group $\mathfrak S_n$ on the  set $\hat \Tf$ of the $\kb^\times $-character of the finite torus $\Tf$. In particular, if $\gamma=\bf 1$ is the orbit of the trivial character, then $\e_{\bf 1}\Hh$ naturally identifies with $\HhI$.
\medskip

Let $\rho$ be an irreducible $k$-representation of $\K$.  Such an object is called a \emph{weight}. It descends to an irreducible representation of $\overline{\mathbf{G}}_{x_0} (\mathbb{F}_q)$ because $\K_1$ is a pro-$p$ group.  Its compact induction to $\Gp$ is denoted by $\ind_{\K}^\Gp \rho$. The $k$-algebra of the $\Gp$-endomorphisms of the latter is denoted by ${\cal H}(\Gp, \rho)$ and will be called the spherical Hecke algebra attached to $\rho$. In \cite{satake}, Herzig describes the algebra  ${\cal H}(\Gp, \rho)$ (remark that the results of \cite{satake} are equally valid when $\Corps$ has characteristic $p$). He proves in particular that it  is a commutative noetherian algebra. For example, if $\mathbf G={\rm GL}_n$ (for  $n\geq 1$) then ${\cal H}(\Gp, \rho)$  is an algebra of polynomials in $n$ variables localized at one of them (Example 1.6, \emph{loc. cit.}). 
More precisely, for general  $\mathbf G$, let $\X_*(\T)$ denote the set of cocharacters of the split torus $\T$  and $\X_*^+(\T)$ the monoid of the dominant ones, then there is an isomorphism $$\EuScript S: {\cal H}(\Gp, \rho)\overset{\simeq}\longrightarrow k[\X_*^+(\T)]$$ 
given by \cite[Thm 1.2]{satake} (see our remark \ref{rema:normalization} for our choice of the \enquote{dominant} normalization).

%Since the work of Buzzard, Diamond and Jarvis \mag{ref}, it has  become common to call \emph{weight} (or \emph{Serre weight}, or \emph{Diamond weight} (\cite{BP}))  an irreducible $k$-representation of $\overline{\mathbf{G}}_{x_0} (\mathbb{F}_q)$. 
%The reason for this choice of vocabulary is that, in the generalizations of Serre's modularity conjecture (for example, to  the mod $p$ representations of the Galois group of a totally real field \cite{BDJ}), these objects play a role that is analogous to the one of the weights -- non negative integers -- attached to a   mod $p$ representation of the Galois group of $\mathbb Q$ in the classical modularity theorem. 
%We also use this denomination here.
%: \green{
%we will call \emph{Weight} an irreducible representation of $\overline{\mathbf{G}}_{x_0} (\mathbb{F}_q)$ with c\oe fficients in $\bar k$.
%We keep the word \emph{coweight}  with no capital letter to denote the coweights attached to an affine root datum.}
\medskip
\subsection{{Results}} 

\subsubsection{}  Let $\rho$ be a weight. We prove independently from \cite{satake} that there is an isomorphism between $k[\X_*^+(\T)]$ and   ${\cal H}(\Gp, \rho)$ 
(depending on the choice of a uniformizer $\varpi$ and of a set of positive roots) by constructing a map in the opposite direction 
\begin{equation}\EuScript T: k[\X_*^+(\T)]\overset{\simeq}\longrightarrow {\cal H}(\Gp, \rho)\end{equation}

\noindent and proving that it is an isomorphism (Theorem \ref{theomy}).  
\medskip

Under the hypothesis that the derived subgroup  of $\mathbf G$ is simply connected, we give in \ref{expli} an explicit description of $\EuScript T$ and prove that it is an inverse for 
$\EuScript S$ which, under the same hypothesis, is explicitly computed in \cite{Parabind}.

\bigskip

Our method to construct $\EuScript T$ is based on the following result: it is well known that there is a one-to-one correspondence between the weights and the characters  of the (finite dimensional) pro-$p$ Iwahori-Hecke algebra $\H$ of the maximal compact $\K$ (\cite{CL}). In fact, we have more than this: by a theorem of Cabanes  (\cite{Cabanes}, recalled in \ref{categories}), there is an equivalence of categories between $\H$-modules and a certain category (denoted here by $\BB(x_0)$) of representations of $\K$. Using this theorem, we prove (Corollary \ref{CoroIso}) that  passing to $\I$-invariant vectors gives  natural  isomorphisms of $k$-algebras

\begin{equation}\label{i0}\cal H (\Gp, \rho )\cong \Hom_\Hh((\ind_{\K }^\Gp \rho)^\I,(\ind_{\K }^\Gp \rho)^\I )\cong\Hom_\Hh(\chi\otimes_\H\Hh,\chi\otimes_\H\Hh )\end{equation}\\
where $\chi$ is the character of $\H$ corresponding to $\rho$. It therefore remains to describe  the $k$-algebra  $\Hom_\Hh(\chi\otimes_\H\Hh,\chi\otimes_\H\Hh )$. The necessary tools are introduced in Section \ref{main}.

\medskip

Let $F$ be a standard facet, that is to say a facet of the standard chamber $C$ containing $x_0$ in its closure.
In Section \ref{main}, we attach to $F$ a Weyl chamber $\Cute^+(F)$ (for example, if $F=C$, it is simply the set of dominant cocharacters) as well as a $k$-linear  \enquote{Bernstein-type} map
$$ \Bb_F^+: k[\tilde\X_*(\T)]\rightarrow \Hh$$ defined on the group algebra of the extended cocharacters $\tilde\X_*(\T)$ (\S\ref{notations:tame} and Remark \ref{rema:tilda}). Restricted to the dominant monoid $k[\tilde\X_*^+(\T)]$, the map $\Bb_F^+$ respects the product and its image is a commutative subalgebra of $\Hh$. For example,  if $F=C$, then $\Bb_C^+$  coincides  on $k[\tilde\X_*^+(\T)]$ with the map carrying a dominant 
cocharacter  onto the  characteristic function of  the corresponding  double coset modulo $\I$, and  the image of $\Bb^+_C$ coincides with the subalgebra $ \mathcal\A^{+, (1)}$ of \cite{Ann}.

\medskip

To the character $\chi$ of $\H$  we attach a  standard facet   $F_\chi$  as well as its  restriction $\bar \chi$ to the finite torus  $\overline {\mathbf T}(\mathbb F_q)$ (\S\ref{param}).
(For example, if $\rho$ is the Steinberg representation, then $F_\chi= C$ and $\bar \chi$ is the trivial character.)
We prove in \ref{subsec:satake} that the map $\Bb_{F_\chi}^+$ induces  an isomorphism  of $k$-algebras

\begin{equation}\begin{array}{ccc}
k[\X_*^+(\T)]\:\cong\:\bar\chi\otimes_{k[\T^0/\T^1]}k[ \tilde \X^+_*(\T)]&\overset{\simeq}\longrightarrow &\Hom_\Hh(\chi\otimes_\H\Hh,\chi\otimes_\H\Hh )\cr\end{array}
\end{equation}   which, combined with \eqref{i0} yields the isomorphism $\EuScript T$.

\subsubsection{}

The classical Satake transform is  an isomorphism (\cite{Sat}, see also \cite{Gross})
$$S: \C[\K\backslash \Gp/\K]\overset{\simeq}\longrightarrow (\C[\X_*(\T)])^\Wf $$ where $\Wf$ denotes the finite Weyl group  corresponding to $\T$.
On the other hand, the  center $\cal Z(\cal H_\C(\Gp, \Iw))$  of the complex Iwahori-Hecke algebra $\cal H_\C(\Gp, \Iw)$ was described
by Bernstein (see  \cite{LuSing}, \cite{Lu}):  it is isomorphic to the algebra of $\Wf$-invariants $(\C[\X_*(\T)])^\Wf$. 
By  \cite[Proposition 10.1]{Haines} (see also \cite[Corollary 3.1]{Dat}),
 the  \emph{Bernstein isomophism} $$B: (\C[\X_*(\T)])^\Wf\overset{\simeq}\longrightarrow \cal Z(\cal H_\C(\Gp, \Iw))$$
is compatible with $S$ in the sense that 
the composition $(e_\K\star.) B$ is an inverse for $S$, where $(e_\K\star.) $ is the convolution  by the characteristic function of $\K$.
\medskip

The maps $\Bb_F^+$ introduced in the present article are modified (and integral) versions of the Bernstein maps used to define the isomorphism $B$ in the complex case for the Iwahori-Hecke algebra.  Therefore, our  construction of an inverse Satake isomorphism  using commutative subalgebras of the (pro-$p$) Iwahori-Hecke algebra  has hints of   a compatibility between mod $p$ Satake isomorphism and Bernstein maps.  
The   link between spherical Hecke algebras and the \emph{center} of the pro-$p$ Iwahori-Hecke algebras  is  further analyzed in a separate paper (\cite{compa}).

\subsection{Acknowledgments} 
 I  thank  Marc Cabanes and Florian Herzig for their {thorough, helpful comments,  and Michael Harris for  energizing  and insightful conversations over the years.

\section{Affine root systems and associated Hecke rings}
We first give notations and basic results about \enquote{abstract} reduced  affine root systems: in the first paragraph the symbols used are underscored. In \ref{subsec:rootdataG}
(respectively  \ref{rootF}), we will describe some aspects of the construction of the reduced affine root system for $\Gp$ (respectively, for a semi-standard Levi subgroup of $\Gp$) associated to the choice of the torus $\T$. In both the settings of  \ref{subsec:rootdataG} and of
\ref{rootF}, the results of \ref{rootgene} apply. 

\subsection{Affine root system\label{rootgene}}

We refer to \cite[\S 1]{Lu} as a general reference. 
We consider  an affine root system
$(\Root, {\rm X}^* \Coroot, {\rm X}_*)$ where $\Root$ is the set of roots and $\Coroot$ the set of coroots. We suppose that this root system is reduced.
An element of the free abelian group ${\rm X}_*$   is  called a \emph{coweight}.
We denote by $\lp\, .\, ,\,.\rp$ the perfect pairing on ${\rm X}_*\times {\rm X}^*$
and by $\alpha  \leftrightarrow \check\alpha $  the correspondence between roots and coroots satisfying
$\lp \alpha ,\check\alpha \rp=2$. We choose a basis $\Pig$  for $\Root$ and denote by
 $\Root^+$ (resp. $\Root^-$)  the set of roots which are positive (resp. negative) with respect to $\Pig$. There is  a partial order on $\Root$ given by $\root\preceq \broot$ if and only if $\broot-\root$ is a linear combination with (integral) nonnegative c\oe fficients of elements in $\Pi$.

%\mag{Let $\{\check\upomega_\root\}_{\root\in \Pi}$  denote the set of fundamental coweights (\cite[Ch VI, 1.10]{Bki-LA}): it satisfies 
%$\lp\check\upomega_\root,\beta\rp=\delta_{\alpha, \beta}$ for all $\alpha, \beta\in\Pi$.}
\medskip

To the  root $\check\alpha $ corresponds the  reflection $s_\alpha  : \lambda\mapsto \lambda-\lp \lambda,\root\rp\,\coroot$ defined on $\X_*$. It 
 leaves $\Coroot$
%$\check\Root$
 stable. 
 %Reciprocally, we will denote by $\root_s$ the simple root associated to the simple reflection $s$.
The finite Weyl group $\Wfg$ is the subgroup of ${\rm GL}(\X_*)$ generated by the simple reflections $s_\root$ for  $\root\in\Pi$.     
It is a Coxeter system with generating set $\Sg=\{s_\root,\:\root\in\Pig\}$. We will denote by $(w_0,\lambda)\mapsto \wl$ the natural action of $\Wfg$ on
the set of coweights. It induces a natural action of $\Wfg$ on the weights which stabilizes the set of roots. 
The set $\X_*$ acts on itself by translations: for any coweight $\lambda\in \X_*$,
we denote by $e^\coweight$ the associated translation.
The (extended) Weyl group  $\Wg$ is the semi-direct product  $\Wfg\ltimes\X_*$.  

\medskip
\subsubsection{}  Define the set of affine roots by  ${\Root_{aff}}=\Root\times \Z={\Root_{aff}^+}\coprod {\Root_{aff}^-}$ where
$${\Root_{aff}^+}:=\{(\root, r),\: \root\in\Root, \,r>0\}\cup\{(\root,0),\, \root\in\Root^+\}.$$
%\: {\Root_{aff}^-}:=\{(\root, r),\: \root\in\Root, \,r<0\}\cup\{(\root,0),\, \root\in\Root^-\}.$$

The Weyl group  $\Wg$ acts on $\Root_{aff}$ by $we^\lambda: (\root, r)\mapsto (w\root,\,\, r-\lp \lambda, \root\rp)$ where we denote by $(w,\root)\mapsto w\root$ the natural action of $\Wfg$ on the roots.
Denote by $\Pig_m$ the set of roots that are minimal elements  for $\preceq$.
Define the set of simple affine roots by  $\Pig_{aff}:=\{(\root, 0),\: \root\in\Pig\}\cup\{(\root,1),\, \root\in\Pig_m\}$. Identifying $\root$ with $(\root,0)$, we consider $\Pig$ a subset of
$\Pig_{aff}$.  For $A\in  \Pig_{aff}$, denote by $s_A$ the following associated reflection: $s_A=s_\root$ if $A=(\root, 0)$ and  $s_A=s_\root e^{\coroot}$ if $A=(\root,1)$.  The length on the Coxeter system $\Wfg$ extends to $\Wg$ in such a way that,   the length of $w\in \Wg$  
 is the number of affine roots   $A\in{\Root_{aff}^+}$ such that
$w(A)\in { \Root_{aff}^-}$. It satisfies  the following formula, for every $A \in \Pig_{aff}$ and $w\in \Wg$:
\begin{equation}\label{add}
   \ellg(w s_A)= 
   \begin{cases}
       \ellg(w)+1 & \textrm{ if }w (A)\in {\Root_{aff}^+},\\  \ellg(w)-1 & \textrm{ if }w (A)\in {\Root_{aff}^-}.
    \end{cases}
\end{equation}\\

The affine Weyl group is defined as the subgroup
$\Wg_{aff} := \; < s_A , \:\: A \in \Root_{aff}>$
of $\Wg$. Let $  \Sg_{aff} := \{s_A : A \in \Pig_{aff}\}$. The pair $(\Wg_{aff}, \Sg_{aff})$ is a Coxeter system (\cite[V.3.2 Thm.\ 1(i)]{Bki-LA}), and the length function $\ellg$ restricted to $\Wg_{aff}$ coincides with the length function of this Coxeter system. Recall  (\cite[1.5]{Lu}) that $\Wg_{aff}$ is a normal subgroup of $\Wg$:   the set $\Omega $ of elements  with length zero  is an abelian subgroup of $\Wg$ and 
 $\Wg$ is the semi-direct product $\Wg= \Omega \ltimes \Wg_{aff}$.
 The length $\ellg$ is constant on the double cosets $\Omega w \Omega$ for $w \in \Wg$. In particular $\Omega$ normalizes $\Sg_{aff}$.

\medskip

We extend the Bruhat order $\leq $ on the Coxeter system $(\Wg_{aff}, \Sg_{aff})$   to $\Wg$ by defining
\begin{center} $\omega_1 w_1\leq \omega_2 w_2$ if $\omega_1=\omega_2$ and $w_1\leq w_2$\end{center} for $w_1, w_2\in \Wg_{aff}$ and $\omega_1, \omega_2\in \Omega$ (see \cite[\S 2.1]{Haines}). We write $w<w'$ if $w\leq w'$ and $w\neq w'$ for   $w,w'\in \Wg$. Note that $w\leq w'$ and $\ellg(w)=\ellg(w')$ implies $w=w'$.

\medskip

\subsubsection{}  Let $\X_*^+$ denote the set of dominant coweights that is to say the subset of all $\lambda\in \X_*$ such that
$$\lp\lambda, \root\rp\geq 0 \textrm{  for all $\root\in \Root^+$}.$$  The set of antidominant coweights is $\X_*^-:= -\X_*^+$.
It is known that the extended Weyl group $\Wg$ is the disjoint union of all $\Wfg e^\lambda \Wfg$ where $\lambda$ ranges over $\X_*^+$ (resp. $\X_*^-$) (see for example \cite[2.2]{Kato-Sph})
\begin{rem}  We have $\ellg(we^\lambda)=\ellg(w)+ \ellg(e^\lambda)$ for all $w\in \Wfg$ and $\lambda\in \X_*^+$. 
\label{eq:length-domi}
\end{rem}

There is a partial order on $\X_*^+$ given by $\lambda\preceq \mu$ if and only if $\lambda-\mu$ 
is a non-negative integral linear combination of the simple coroots.

%\begin{rem}\label{rema:important}  \mag{ For  $\mu$ and $\nu$  both dominant!!!! $\mu{\preceq}\lambda$ means $e^\mu  \leq e^\lambda$ et Lusztig est ecrit pour le cas (semi)-simple}. 

%\end{rem}

\subsubsection{Distinguished coset representatives\label{disting}}

The following statement is \cite[Proposition 4.6]{OS} (see \cite[Lemma 2.6]{Oparab} for ii).

\begin{pro}\label{prop:D} Let $\Ddg$ be the subset of 
the elements $d$ in $\Wg$ such that  \begin{equation}d^{-1}\Root^+\subset \Root_{aff}^+.\end{equation} 
\begin{enumerate}
\item It   is a system of representatives  of  the right cosets ${\Wfg}\backslash \Wg$. Any  $d\in \Ddg$ is the unique element with minimal length in  $\Wfg d$ and for any $w\in \Wfg$, we have
\begin{equation}\ellg(wd)=\ellg(w)+ \ellg(d).\label{addi}\end{equation}
\item An element  $d\in \Ddg$ can be written uniquely $d=e^\lambda w$ with $\lambda\in \X_*^+$ and  $w\in \Wfg$. We then have $\ellg(e^\lambda)=\ellg(d)+\ellg(w^{-1})$.  
\item Let $s\in \Sg$ and $d\in \Ddg$.  If $\ellg(ds)=\ellg(d)-1$ then $ds\in \Ddg$.
  If $\ellg(ds)=\ellg(d)+1$ then either $ds\in \Ddg$, or $ds\in \Wfg d$. 

\end{enumerate}
\end{pro}

\begin{rem} \label{rema:D}Let $\lambda\in \X^+_*$. \\  - Then $e^\lambda\in \Ddg$ and $\Ddg\cap \Wfg e^\lambda \Wfg=\Ddg\cap e^\lambda \Wfg.$\\- There is a unique element with maximal length in $\Wfg e^\lambda \Wfg $: it is
 $w_\lambda:=w_0 e^\lambda$  where  $w_0$ is the unique element with maximal  length in $\Wfg$.

\end{rem}

\begin{lemma}\label{photo}
Let $\lambda, \mu\in \X^+_*$  and $d\in \Ddg\cap e^\lambda\Wfg$. 
\begin{enumerate}

\item
$d\leq e^\lambda$   and in particular $\ellg(d)<\ellg(e^\lambda)$ if $d\neq e^\lambda$.

%\item If $d\neq e^\lambda$, then there exists $s\in \Sg$ such that $\ell(ds)=\ell(d)+1$ and $ds\in \Ddg$.
\item   $d\leq e^{\mu}$  is equivalent to $e^\lambda\leq e^{\mu}$.
\item Let $w\in {\Wfg} e^{\lambda}  {\Wfg}$. If $w\leq w_{\mu}$
 then $e^\lambda\leq e^{\mu}$. In particular, $w_\lambda\leq w_\mu$ is equivalent to $e^\lambda\leq e^\mu$.

\end{enumerate}
\end{lemma}

\begin{proof}
The first assertion comes from ii. of Proposition \ref{prop:D}. To prove the second assertion,
write   $d= e^{\lambda} w$ with $w\in \Wfg$ and suppose that $d\leq e^{\mu}$.  
If  $w\neq 1$,   then
 ${}^{w^{-1}}\!\lambda$  is not a  dominant coweight otherwise by  Remark \ref{eq:length-domi} we would have
$\ellg(d)>\ellg(e^\lambda)$. %\ellg(w)+ \ellg(e^{{}^{w^{-1}}\!\lambda})=\ellg(w)+ \ellg(e^{\lambda})$. 
Therefore, there is
$\beta\in \Pi$  such that $\lp {}{}^{w^{-1}}\!\lambda,\beta\rp<0$, that is to say $d(\beta,0)=(w \beta, -\lp {}^{w^{-1}}\!\lambda,\beta\rp)\in \Root_{aff}^+-\Root^+$. This implies that $\ellg(ds_\beta)=\ellg(d)+1$ by \eqref{add} and that $ds_\beta d^{-1}\not\in \Wfg$ so that $ds_\beta\in \Ddg$ after Proposition \ref{prop:D} iii. 
Note that applying Proposition \ref{prop:D} ii.  to $d$ and $ds_\beta$ shows that
 $\ellg(ws_{\beta})=\ellg(w)-1$.   
By  Lemma \cite[4.3]{Haines} (repeatedly) we get from $d\leq e^\mu$ that  $ds_\beta\leq e^{\mu}$ (we have either $ds_{\beta}\leq e^{\mu}$ or $ds_\beta\leq e^{\mu} s_\beta$. In the latter case, 
$ds_\beta\leq e^{\mu} s_\beta\leq e^{\mu}$  if $\lp\mu, \beta \rp>0$ ; otherwise $\lp\mu, \beta \rp=0$ and $e^{\mu}$ and $s$ commute: we have $ds_\beta\leq  s _\beta e^{\mu} $ which  implies
 that either $ds_\beta \leq e^{\mu}$ or   $s_\beta ds_\beta \leq e^{\mu}$, but $ds_\beta\leq s_\beta ds_\beta$ because $ds_\beta\in \Ddg$, so in any case $ds_\beta\leq e^{\mu}$).
We  then complete   the proof of the second assertion  by induction on $\ellg(w)$.
 \\
To prove the last assertion,  it is enough to consider the case $w=d\in \Ddg$. We prove by induction on $\ellg(u)$ for $u\in \Wfg$ that $d\leq ue^{\mu}$ implies  $d\leq  e^{\mu}$: let $s\in S$ such that $\ellg(su)=\ellg(u)-1$;  by Lemma \cite[4.3]{Haines} we have $d\leq sd\leq su e^{\mu}$ or
$d\leq su e^{\mu}$; conclude. Therefore, $d\leq w_{\mu}$ implies $d\leq  e^{\mu}$ and by ii., $e^{\lambda}\leq e^{\mu}$. \end{proof}

One easily  deduces from the previous Lemma (see also \cite[\S1]{LuSing} for the compatibility between the partial orderings $\preceq$ and $\leq$ on the dominant coweights)  the following well known result (\cite[7.8]{HKP}, \cite[(4.6)]{Kato-Sph}).  Let  $\lambda \in \X_*^+$.  We have \begin{equation}\label{connu}\{w\in \Wg, \:w{\leq} w_\lambda \}=\coprod_{\mu} \Wfg e^{\mu }\Wfg\end{equation}
where $\mu\in \X_*^+$ ranges over the  dominant coweights such that  
$e^\mu\leq e^\lambda$ or equivalently $\mu{\preceq}\lambda$.

%\begin{proof}
%Let  $\mu\in \X_*^+$.  If  $\mu{\preceq}\lambda$, then $e^\mu  \leq e^\lambda$ (Remark \ref{rema:important})
%and therefore $w_\mu\leq w_\lambda$. By Lemma \ref{photo} i., an element $w\in \Wfg e^{\mu }\Wfg$  satisfies $w\leq w_\mu$
%which gives one the required inclusions. The other inclusion is Lemma \ref{photo} iii.
%\end{proof}

\subsection{\label{subsec:rootdataG}Affine root system  attached to $\mathbf{G}(\Corps)$}
We refer  for example  to \cite[I.1]{SS}  and \cite{Tit} for the description of the root datum $(\Phi, {\rm X}^*(\Tp), \check\Phi, {\rm X}_*(\Tp))$
associated to the choice  (\S\ref{lasko}) of  a maximal $\mathfrak F$-split torus $\Tp$ in $\Gp$ (or rather, $\Tp$ is the group of $\Corps$-points  of a maximal torus in $\mathbf G$). 
This root system is reduced because  the group $\mathbf{G}$ is $\mathfrak F$-split.

\subsubsection{Apartment attached to a maximal split torus\label{rootG}}

The set  ${\rm X}^*(\Tp)$ (resp.   ${\rm X}_*(\Tp)$) is   the set of algebraic characters (resp. cocharacters)
of $\Tp$.  The cocharacters will also be called the  coweights.  Let $\X^*({\rm Z})$ and $\X_*({\rm Z})$ denote respectively the set of algebraic characters and cocharacters
of the connected center ${\rm Z}$ of $\Gp$.

 As before, we denote by
$ \lp \, .\,,.\, \rp: {\rm X}_*(\Tp)\times {\rm X}^*(\Tp)\rightarrow \Z$
the natural perfect pairing. 
$\lp \, .\,,.\, \rp$.  Its $\mathbb R$-linear extension is also denoted by
$\lp \, .\,,.\, \rp$ .
\medskip
The vector space
\begin{equation*}
    \mathbb R\otimes _{\mathbb Z}(\X_*({\Tp})/\X_*({\rm Z})) 
    %\ ,\ \text{resp.}\ \mathbb R\otimes _{\mathbb Z} X_*({\Tp}) \ ,
\end{equation*}
considered as an affine space  on itself identifies with an apartment $\Aa$ of the building $\mathscr{X}$ that we will call standard. We choose the hyperspecial vertex $x_0$ as an origin of $\Aa$. Note that the corresponding  apartment in the extended building $\mathscr X^{ext}$ as described in \cite[4.2.16]{Tit} is the affine space
  $\mathbb R\otimes _{\mathbb Z}\X_*({\Tp})$.  
Let $\root\in \Phi$. Since  $\lp ., \root\rp$ has value zero on $\X_*({\rm Z})$, the function $\root(\,.\,):=\lp\, . \,, \root\rp$ on 
$\Aa$ is well defined. The reflection  $s_\root$   associated to a root $\root\in\Phi$ 
can be seen as a reflection on  the affine space $\Aa$ given by
    $s_\root: x\mapsto x- \root(x)  \coroot $.
    The natural action on $\Aa$ of the  normalizer  $N_\Gp(\Tp)$ of $\T$  in $\Gp$  yields an isomorphism between
   $N_\Gp(\T)/\T$ and  the subgroup $\Wf$ of the transformations of  $\Aa$ generated by these reflections. 

The  choice of the chamber $C$ (\S\ref{lasko}) of the standard apartment implies in particular the choice of   the subset
 $\Phi^+$   of the positive  roots,  that is  to say the set of all $\root\in \Phi$ that take non negative values on $C$.  Set $\Phi^-:= -\Phi$. We fix  $\Pi$ a basis for $\Phi^+$.  We denote by $\Phi_{aff}$ (resp. $\Phi_{aff}^+$, resp. $\Phi_{aff}^-$) the set of affine (resp. positive affine, resp. negative affine)  roots, and by $\Pi_{aff}$ the corresponding basis for $\Phi_{aff}$ as in \ref{rootgene}.
 Denote by $\X_*^+(\T)$ (resp.  $\X_*^-(\T)$) the set of dominant (resp. antidominant) coweights.
 The partial ordering on $\X_*^+(\T)$ associated to $\Pi$ is denoted by $\preceq$.

% The finite Weyl group $\Wf$ is a Coxeter system generated by  the set $S:= \{s_\root : \root \in \Pi\}$ of reflections associated to the simple roots $\Pi$. It is endowed with a length function denoted by $\ell$.
%There is  a partial order on $\Phi$ given by $\root\preceq \broot$ if and only if $\broot-\root$ is a linear combination with (integral) nonnegative c\oe fficients of elements in $\Pi$. 
%We set   $ \Phi^{min} := \{\root \in \Phi,\: \root\ \textrm{is minimal for $\preceq$}\}$.

\medskip

The extended Weyl group $\W$  is the semi-direct product of $\Wf\ltimes\X_*(\T)$. It contains the affine Weyl group $\W_{aff}$.  We denote by $\ell$ the length function and by $\leq$ the Bruhat ordering on $\W$. They extend the ones on the Coxeter system $(\W_{aff}, \SS_{aff})$.

\medskip

To an element $g\in \Tp$ corresponds a vector $\nu(g)\in \mathbb R\otimes _{\mathbb Z}\X_*({\Tp})$ defined by
\begin{equation}\label{normalization}
    \lp\nu(g),\, \chi\rp =-\val_{\mathfrak F}(\chi(g))  \qquad \textrm{for any } \chi\in \X^*(\Tp).
\end{equation}
The kernel of $\nu$ is the maximal compact subgroup $\T^0$ of $\Tp$. The quotient of $\Tp$ by $\T^0$  is a free abelian group  with rank equal to $\rm dim(\Tp)$, and $\nu$ induces an isomorphism  $\Tp/\T^0\cong\X_*(\Tp)$. The group $\Tp/\T^0$ acts by translation on $\Aa$ via $\nu$.
The actions of $\Wf$ and $\Tp/\T^0$ combine into an action of 
the quotient of $N_\Gp(\Tp)$ by $\Tp^0$  on $\Aa$ as recalled in \cite[page 102]{SS}. Since $x_0$ is a special vertex of the building,   this quotient identifies with $\W$ (\cite[1.9]{Tit}) and from now on we identify $\W$ with 
 $N_\Gp(\Tp)/\Tp^0$. In particular, a simple reflection $s_A\in \SS_{aff}$ corresponding to the affine root $A=(\root, r)$ can be seen as the reflection at the hyperplane with equation $\lp\,.\,, \root\rp+r=0$ in the affine space $\Aa$.

 \medskip

 We denote by $\Dd$ the distinguished set of representatives of the cosets $\Wf\backslash \W$ as defined  in \ref{disting}.

\medskip

\begin{rem} \label{rema:normalization}In \cite{satake} the chosen isomorphism between $\T/\T^0$ and $\X_*(\T)$ is not the same as   \eqref{normalization}. Here we  chose to follow \cite[1.1]{Tit} and \cite[I.1]{SS}.
The consequence is that the image in $\T/\T^0$ of the submonoid $\T^-:=\{t\in \T, \: \val_\Corps(\root(t))\leq 0 \textrm{ for all }\root\in \Phi^+\}$
(cf \cite[Definition 1.1]{satake})  corresponds, in our normalization, to the submonoid $\X_*^+(\T)$ of $\X_*(\T)$. In explains why the dominant coweights appear naturally in our setting.

\end{rem}

\subsubsection{\label{notations:tame}Tame extended Weyl group}
Let $\Tp^1$ be the pro-$p$ Sylow subgroup of $\T^0$.
We denote by $\tilde\W$ the quotient of $N_\Gp(\Tp)$   by $\Tp^1$  and obtain the exact sequence
$$0\rightarrow \Tp^0/\Tp^1 \rightarrow \tilde\W\rightarrow \W\rightarrow 0.$$ We  fix a lift $\tilde{w} \in \tilde\W$ of any $w \in \W$.  

\medskip
The length function $\ell$ on $\W$ pulls back to a length function $\ell$ on $\tilde\W$  (\cite[Prop. 1]{Ann}).
For $u,v\in \tilde\W$ we write $u\leq v$ if their projections $\bar u$ and $\bar v$ in $\W$ satisfy $\bar u\leq \bar v$.

\medskip

For any $A\subseteq \W$ we denote by $\tilde A$ its preimage in $\tilde \W$. In particular,
we have the set $\tilde \X_*(\T)$: as well as those of $\X_*(\T)$, its elements will be denoted by $\lambda$ or $e^\lambda$. For $\root\in \Phi$, we inflate the function $\root(\,.\,)$ defined on $\X_*(\T)$ to 
$\tilde \X_*(\T)$.  We will  write $\lp x, \root\rp:=\root(x)$ for  $x\in \tilde \X_*(\T)$.
We still call \emph{dominant coweights} the elements in the preimage  $\tilde \X_*^+(\T)$ of $\X_*^+(\T)$.
\subsubsection{\label{bruhat}Bruhat decomposition} 
We have the decomposition $\Gp = \Iw N_\Gp(\rm T)\Iw$ and two cosets $ \Iw n_1\Iw$ and $ \Iw n_2\Iw$  are equal if and only if $n_1$ and $n_2$ have the same projection in $\W$. 
In other words,  a system of representatives in $N_\Gp(\rm T)$ of the elements in $\W$  provides a system of representatives of the double cosets of $\Gp$ modulo $\Iw$. This follows from \cite[3.3.1]{Tit}. We fix  a lift $\hat{w} \in N_\Gp(\Tp)$ for any $w \in \W$  (resp. $w \in \tilde\W$).
In \ref{rootsubgroup} we will introduce specifically chosen lifts for the elements $\tilde s$,  where $s\in \SS_{aff}$. By \cite[Theorem 1]{Ann} the group $\Gp$ is the disjoint union of the double cosets $\I\hat w \I$ for all $w\in \tilde \W$.

\begin{rem} For $w\in \tilde \W$, we will sometimes write $w \I w^{-1}$ instead of $\hat w \I \hat w^{-1}$ since it does not depend on the chosen lift.
\end{rem}

\subsubsection{\label{cartan}Cartan decomposition} 
The  double cosets of $\Gp$ modulo $\K$ are indexed by the coweights in a chosen Weyl chamber:  for $\lambda\in \X_*^+(\T)$, the element $\lambda(\varpi)$ is a lift for $e^{-\lambda}\in \W$ (see Remark \ref{rema:normalization}) and $\Gp$ is the disjoint union of the double cosets $\K \lambda(\varpi) \K$ for $\lambda\in \X_*^+(\T)$.

\subsubsection{Root subgroups\label{rootsubgroup} and Chevalley basis}

For $\alpha\in \Phi$, we consider the   attached  unipotent subgroup  $\Uu _\alpha$  of $\Gp$ as in  (\cite[6.1]{BTI}). To an  affine root  $(\alpha, r)\in \Phi_{aff}$  corresponds a subgroup
$\Uu _{(\alpha, r)}$ of $\Uu _\alpha$ (\cite[1.4]{Tit}) the following properties of which we are going to use (\cite[p. 103]{SS}):\\
- For $r\leq r'$ we have $\Uu _{(\alpha, r')}\subseteq \Uu _{(\alpha, r)}$.\\
- For $w\in \W$, 
the group $\hat w \Uu _{(\alpha, r)}\hat w^{-1}$ does not depend on the lift $\hat w\in\Gp$ and is equal to $\Uu _{w(\alpha, k)}$.

\medskip

 We fix an épinglage for $\Gp$ as in SGA3 Exp. XXIII, 1.1 (see \cite{Conrad}). In particular, to $\root\in\Phi$ is attached a central isogeny $\phi_\root: {\rm SL}_2(\Corps)\rightarrow \Gp_\root$ where $\Gp_\root$ is the subgroup of $\Gp$ generated by $\Uu_\root$ and $\Uu_{-\root}$ (\cite[Thm 1.2.5]{Conrad}). 
 %We have $\phi_\alpha(\begin{pmatrix}1&u\cr 0&1\end{pmatrix})\in \Uu_{(\root, 0)}$ and  $\phi_\alpha(\begin{pmatrix}1&0\cr u&1\end{pmatrix})\in \Uu_{(-\root, 0)}$ for all $u\in \Oo$ (check).  
 \medskip

 We set $n_{s_{\root}}:= \phi_\alpha\begin{pmatrix}0&1\cr -1&0\end{pmatrix}$ and, for $u\in \Corps^*$,  $h_\root(u):=\phi_\alpha\begin{pmatrix}u&0\cr 0&u^{-1}\end{pmatrix}.$ Then $\T$ contains $h_\root(\Corps^*)$ for all $\root\in\Phi$. After embedding $\fq^*$ into $\Corps ^*$ by Teichm{\"u}ller lifting, we  consider the subgroup  $\T_\root(\fq)$  of $\T$ equal to the  image 
of $\fq^*$ by $h_\root$.

 For  $\root\in\Pi_m$, set $h_{(\root, 1)}:= h_\root$, $\T_{(\root, 1)}(\fq):=  \T_{\root}(\fq)$  and
 $n_{s_{(\root, 1)}}:= \phi_\alpha\begin{pmatrix}0&\varpi\cr -\varpi^{-1}&0\end{pmatrix}$.

 \medskip
 
For $A\in \Pi_{aff}$, 
%we denote by $\sigma_A$ the image of $n_{s_A}$ in $\tilde \W= N_G(\T)/ \T^1$. 
the element  $n_{s_A}\in  N_G(\T)$ is a lift for  $s_A\in \SS_{aff}$ (\cite[Proof of Proposition 1.3.2]{Conrad}).
The normalizer $N_\Gp(\T)$ of $\T$ is generated by $\T$ and all the $n_{s_{\root}}$ for $\root\in \Phi$.
For all $w\in \tilde \W$ with length $\ell$, there is  $\omega\in \tilde \W$ with length zero and $s_1, ..., s_\ell\in \SS_{aff}$ such that the product
$n_{s_1}... n_{s_\ell}\in N_\Gp(\T)$ is a lift for $\omega w\in \tilde \W$.

\subsection{\label{rootF}Affine root system attached to a standard facet}

\subsubsection{} Let $F\subseteq \overline C$ be a facet containing $x_0$ in its closure.   Such a facet will be called \emph{standard}. Attached to it is the subset $\Pi_F$ of the roots in $\Pi$ taking value zero on $F$, or equivalently the subset $\SS_F$ of the reflections in $\SS$ fixing $F$ pointwise. 

\begin{rem}\label{facet}
The closure $\overline{F}$ of a facet $F$ consists exactly of the points of $\overline C$ that are fixed by the reflections in $\SS_F$ (\cite[V.3.3 Proposition  1]{Bki-LA}).
\end{rem}

We let $ \Phi_F$ denote the set of  roots in $\Phi$ taking value zero on $F$
and set $ \Phi_F^+ := \Phi_F \cap \Phi^+$,  $\Phi_F^- : = \Phi_F \cap \Phi^-$.
We consider the root datum $(\Phi_F, {\rm X}^*(\Tp), {\check\Phi}_F, {\rm X}_*(\Tp))$. 
The    corresponding finite Weyl group  $\Wf_F $ is the subgroup of $\Wf$ generated by all $s_\root$ for  $\root\in  \Phi_F$.
The pair $(\Wf_F, \SS_F)$ is a Coxeter system. The restriction $\ell | \Wf_F$ coincides with its length function (\cite[IV.1.8 Cor.\ 4]{Bki-LA}).
 The extended Weyl group is $\W_F=\Wf_F\ltimes\Weight(\Tp)$. Its action on the  affine roots are $\Phi_{F, aff}:=\Phi_F\times \Z$ coincides with the restriction of the action of $\W$.
 %We consider the subgroup  of $\W$ generated by $\W_F$ and $\X_*(\Tp)$:
%$\W_F=\Wf_F\ltimes\Weight(\Tp)$.  It is the extended Weyl group associated to the
Denote by $\underset{F}\preceq$ the partial order on $\X_*(\T)$ with respect to $\Pi_F$, by
$\W_{F, aff}$  the affine Weyl group with generating set $\SS_{F, aff}$ defined as in \ref{rootgene}.
% and by  $ \Phi_F^{min}$  the set of $\root \in \Phi_F$ that are  minimal for $\underset{F}\preceq$. Let $\Pi_{F, aff}:=     \Pi_F \cup \{(\root,1), \: \alpha \in \Phi_F^{min}\}$.  
%The affine Weyl group $\W_{F, aff}$ for  the root system associated to $F$ is the subgroup of $\W_F$ generated by all $s_\root$ for $\root\in 
%\Phi_{F, aff}$. The pair $(\W_{F, aff}, S_{F, aff})$ is a Coxeter system where $S_{F, aff}$ is the set of affine reflections associated to roots in $\Pi_{F, aff}$.  
It comes with a length function denoted by $\ell_F$ and a Bruhat order denoted  by $\underset{F}\leq$,  which  can both be extended
to $\W_F$. 
%For  $w\in \W_F$, $\ell_F(w)$ is the number of affine roots   $A\in{\Phi_{F, aff}\cap \Phi_{aff}^+ }$ such that $w(A)\in {\Phi_{F, aff}\cap \Phi_{aff}^- }$ 

\medskip

\subsubsection{\label{subsec:strongly}}  
%The set of $\root \in \Phi_F$ that are  minimal for $\underset{F}\preceq$ is distinct from  $ \Pi_m$  in general, therefore
The restriction $\ell\vert\W_F$ does not coincide with $\ell_F$ in general, and likewise the restriction to $\W_F$  of the Bruhat order on $\W$ does not coincide with $\underset{F}\leq$. 
 We call $F$-positive the elements $w$ in $\W_F$  satisfying
$$w^{-1}(\Phi^+-\Phi_F^+)\subset \Phi_{aff}^+.$$ For $\lambda\in \X_*(\T)$, the element $e^\lambda$ is  $F$-positive
if  $\lp \lambda, \root \rp\geq 0$  for all $\root\in \Phi^+-\Phi_F^+.$ In this case, we will say that the coweight $\lambda$ itself is $F$-positive. We observe that if $\mu$ and $\nu \in \X_*(\T)$ are $F$-positive  coweights such that $\mu-\nu$  is also $F$-positive, then we have the equality.
\begin{equation}\label{lengthF}
\ell(e^{\mu-\nu})+\ell(e^{\nu})-\ell(e^{\mu})=  \ell_F(e^{\mu-\nu})+\ell_F(e^{\nu})-\ell_F(e^{\mu})
\end{equation}
Its  left hand side is  indeed by definition
$\sum_{\root\in \Phi^+}\vert \lp\mu-\nu,\root\rp\vert +\vert \lp \nu,\root\rp\vert -\vert \lp\mu ,\root\rp\vert$
but the contribution to this sum of the roots in $\Phi^+- \Phi^+_F$ is zero since 
$\mu-\nu$, $\mu$ and $\nu$ are $F$-positive.  \\

Since the elements in $\Wf_F$  stabilize the set $\Phi^+-\Phi_F^+$, an element in $\W_F$ is $F$-positive if and only if it belongs to $\Wf_F e^\lambda \Wf_F$ for some $F$-positive coweight $\lambda$. 
The $F$-positive elements in $\W_F$ form a semi-group.
A coweight $\lambda$    is said strongly $F$-positive if  $\lp \lambda, \root \rp> 0$  for all $\root\in \Phi^+-\Phi_F^+$ and
$\lp \lambda ,\root \rp=0$  for all $\root\in\Phi_F^+$. By \cite[Lemma 6.14]{BK}, strongly $F$-positive elements do exist.

\medskip
\begin{rem}
If $F= x_0$, then $\W_{x_0}= \W$.
If $F=C$ then $\W_C= \X_*(\Tp)$ and the $C$-positive elements are the dominant   coweights. A strongly $C$-positive element will be called strongly dominant.
\label{strongly}

\end{rem}

\begin{lemma} \label{F-positive}

{\emph i}. Let $\mu\in\X_*(\T)$ and  $\lambda\in\X_*^+(\T)$ such that
$\mu\preceq _F\lambda$. Suppose that for all $\alpha\in \Phi_F^+$ we have $\lp\mu, \alpha \rp\geq 0$, then $\mu\in\X_*^+(\T)$.\\
\emph{ii}. Let $v\in \W_F$ such that $v\underset{F}\leq e^{\lambda}$ for some $\lambda\in \X^+_*(\T)$. Then $v$ is $F$-positive and
  there is $\mu\in  \X^+_*(\T)$ with $\mu\preceq _F\lambda$ such that  $\Wf_Fv \mathfrak W_F= \Wf_Fe^{\mu}\Wf_F$.

\end{lemma}

\begin{proof}i.   Let  $\alpha\in\Pi\backslash \Pi_F$. For all $\beta\in \Pi_F$, we have  $\lp\check\beta, \alpha \rp\leq 0$ \cite[Thm1 Ch VI, n 1.3]{Bki-LA} so
 $\lp\lambda-\mu, \alpha \rp\leq 0$ and
 $\lp\mu, \alpha \rp\geq 0$.
For ii., note that  in particular, $v\underset{F}\leq w_{F,\lambda}$ where $w_{F,\lambda}$ denotes the element with maximal length in $\Wf_Fe^\lambda \Wf_F$.
By \eqref{connu} applied to the root system associated to $F$,  there is a  unique $\mu\in \X_*(\T)$  with $\lp \mu,\root\rp \geq 0$ for all $\root\in \Phi^+_F$  and $\mu\underset{F}\preceq \lambda$ such that
$v\in \mathfrak W_F e^\mu \mathfrak W_F$, and i. implies that $\mu\in \X_*^+(\T)$.
In particular it is $F$-positive and
$v$ is also $F$-positive.

\subsubsection{\label{LeviDefi}} The affine root datum  $(\Phi_F, {\rm X}^*(\Tp), {\check\Phi}_F, {\rm X}_*(\Tp))$ is in fact the one attached to the semi-standard Levi subgroup $\M_F$ of $\Gp$ corresponding to the facet $F$ described below.

Consider   the subtorus $\T_F$ of $\T$  with dimension 
$dim(\T)-\vert \Pi_F\vert$ equal to the connected component of $\bigcap_{\root\in \Pi_F}\ker \root\subseteq \T$  and the Levi subgroup $\M_F$ of $\Gp$ defined to be the centralizer of $\T_F$. 
It is the group of $\Corps$-points of a reductive connected algebraic group $\mathbf M_F$ which is $\Corps$-split. 
The group $\M_F$ is generated by $\T$ and the root subgroups $\Uu_\root$ for $\root\in \Phi_F$.
%(\green{reference}) Borel Tits page 65

The subgroup $(N_\Gp(\T)\cap \M_F)/\T^0$  of $N_\Gp(\T)/\T^0$ identifies with $\W_F$
in the isomorphism $N_\Gp(\T)/\T^0\simeq \W$.  It is generated by $\T$ and all $n_\root$ for $\root\in \Phi_{F}$.
Denote by $\mathscr X_F^{ext}$ the extended building for $\M_F$. It shares with $\mathscr X^{ext}$ the apartment corresponding to $\T$ but, in this apartment, the set of affine hyperplanes 
coming from the root system attached to $\M_F$ is a subset of those coming from the root system attached to $\Gp$.
Every facet in $\mathscr X^{ext}$ is contained in a unique facet of $\mathscr X_F^{ext}$ \cite[\S 2.9]{HainesBC}. Denote by $\mathbf c_F$
the unique facet in $\mathscr X_F^{ext}$ containing ${\rm pr}^{-1}(C)$.
By \cite[Lemma 2.9.1]{HainesBC}, the intersection $\Iw\cap \M_F$ is an Iwahori subgroup for $\M_F$: it is the pointwise  fixator in $\M_F$  of  $\mathbf c_F$. Its pro-$p$ Sylow subgroup is $\I\cap \M_F$.
We have a Bruhat decomposition for $\M_F$: it is the disjoint  union of the double cosets
$(\Iw\cap \M_F)\hat w (\Iw\cap \M_F)$ where $\hat w$ denotes the chosen lift for $w\in \W_F$ in $\Gp$ (\S\ref{bruhat}) which in fact belongs to $\M_F$.

Denote by $\tilde \W_F$ the quotient  $(N_\Gp(\T)\cap \M_F)/\T^1$. It is generated by $\T^0/\T^1$ and all 
$\tilde w$ for $w\in \W_F$. We have an exact sequence $$0\rightarrow \Tp^0/\Tp^1 \rightarrow \tilde\W_F\rightarrow \W_F\rightarrow 0.$$
The Levi subgroup $\M_F$ is the disjoint union of the double cosets $\I\hat w \I$ for all $w\in \tilde \W_F$.
We denote by $\tilde\Wf_F$ the preimage of $\Wf_F$ in $\tilde \W_F$.

\end{proof}

\subsection{Generic Hecke rings\label{geneHecke}} 

\subsubsection{\label{prezIM}} 
For $g\in \Gp$ we denote by $\t_g$ the characteristic function of $\I g\I$. Since it only depends on the element $w\in \tilde \W$ such that $g\in \I \hat w\I$, we will also denote it by $\t_w$.  
The convolution ring $\Htg$  of the functions with finite support in $\I\backslash \Gp/ \I$ and values in $\Z$  is a free $\Z$-module with basis
the set of all $\{\t_{w}\}_{w\in \tilde \W}$ with product given by
 (\cite[Theorem 1]{Ann}) the following braid and respectively quadratic relations:\\
\begin{equation} \label{braid} \textrm{$\t_{ww'}=\t_w \t_{w'}$ for $w, w'\in \tilde \W$ satisfying $\ell(ww')=\ell(w)+ \ell(w')$.}\end{equation}
\begin{equation} \label{Q}   \textrm{
$\t_{n_A}^2=\nu_A\t_{n_A}+ q\t_{h_A(-1)}$ for $A\in \Pi_{aff}$,}\textrm{ where  $\nu_A:=\sum_{t\in \T_A(\fq)}\t_{t}$.}\end{equation}

The braid relations imply that $\Htg$ is generated by all $\t_{n_A}$ for $A\in \Pi_{aff}$ and $\t_\omega$ for $\omega\in \tilde \W$ with length zero.

\subsubsection{\label{involetal}} 

For any $w\in \W$, define $\tg_w^*$ to be the element in  $\Htg\otimes_\Z \Z[q^{\pm 1/2}]$ equal to $q^{\ell(w)}\tg_w^{-1}$. It actually lies in $\Htg$ and the
ring  $\Htg$ is endowed with an involutive automorphism defined by (\cite[Corollary 2]{Ann}) \begin{equation}\upiota:\tg_w\mapsto (-1)^{\ell(w)}\tg_{w^{-1}}^*.\label{involution}\end{equation}

\begin{rem} \label{invoNA}We have $\upiota(\t_{n_A})=-\t_{n_A}+\nu_A$.

\end{rem}

The following fundamental Lemma is proved in \cite[Lemma 13]{Ann} which is a
adaptation to the pro-$p$ Hecke ring of the analogous results of \cite[\S 5]{Haines} for the Iwahori-Hecke ring.
\begin{lemma} For $v, w\in\tilde \W$ we have in $\Htg\otimes_\Z \Z[q^{\pm 1/2}]$
$$q^{\frac{\ell(vw)+\ell(w)-\ell(v)}{2}}
\t_v\t_{w^{-1}}^{-1}= \t_{vw}+\sum_{x}a_x \t_x$$where $a_x\in \Z$ and $x$ ranges over a finite set of elements in $\tilde\W$ with length $<\ell(vw)$.  More precisely, these elements satisfy   $x<  {vw}$ (see \ref{notations:tame}).

\label{fond}
\end{lemma}
\subsubsection{\label{subsec:fourier}}
Following \cite[\S1.3, page 9]{Ann},  we suppose in this paragraph that $\R$  is a ring containing  an inverse for $(q. 1_\R-1)$ and  a primitive  $(q-1)^{\rm th}$ root of $1$.  
We  denote by $\R^\times$ the group of invertible elements of $\R$.
Recall that $\overline {\mathbf T}(\fq)$ identifies with $\T^0/\T^1$ and can therefore be seen as a subgroup of $\tilde\W$.
The finite Weyl group $\Wf$ identifies with the Weyl group of $\overline{\mathbf G}_{x_0}(\mathbb F_q)$ (\cite[3.5.1]{Tit}): it acts on  $\overline {\mathbf T}(\fq)$ and its $\R$-character. Inflate this action to  an
action of the extended Weyl group $\W$. Let $\xi: \overline {\mathbf T}(\fq)\rightarrow \R^\times$ be a $\R$-character of $\overline{\mathbf T}(\fq)$. We attach to it the following idempotent element 
$$\epsilon_{\xi}:=\dfrac{1}{\vert  \overline {\mathbf T}(\fq)\vert}\sum_{t\in  \overline {\mathbf T}(\fq)}\xi^{-1}(t) \t_t\:\:\in\:\:  \Htg\otimes _\Z \R.$$ Note that for $t\in \overline {\mathbf T}(\fq)$, we have 
$\epsilon_{\xi} \t_{t}=\t_{t}\epsilon_{\xi} =\xi(t)\epsilon_{\xi}.$ It implies that the quadratic relations in 
$ \Htg\otimes _\Z \R$ have the (simpler) form: let  $A\in \Pi_{aff}$

\begin{equation}\label{Q'}\begin{array}{l}
\textrm{- if $s_A. \xi=\xi$  then $\xi(h_A(-1))=1$  and   $\epsilon_{\xi}\t_{n_A}^2=\epsilon_{\xi}((q-1) \t_{n_A}+q)$.}\cr
\textrm{- if $s_A. \xi\neq \xi$ then $\epsilon_{\xi}\t_{n_A}^2=q\epsilon_{\xi}\xi({h_A(-1)})$}.\cr\end{array}\end{equation}

\subsubsection{\label{HeckeRingF}}  Let $F$ be a standard facet. The definitions of  the previous paragraphs apply to  the Levi subgroup $\M_F$ and its root system (\S\ref{rootF}). 
In particular, for $w\in \tilde \W_F$,  denote by  $\t_{w}^F$ the characteristic function of $(\I \cap \M_F) \hat w (\I \cap \M_F)$ and by $\tilde {\rm H}_\Z({\M_F})$ the  Hecke ring  as defined in \ref{prezIM}. It has $\Z$-basis   the set of all $\t_{w}^F$ for $w\in \tilde\W_F$ and the braid relations are  controlled by the length function $\ell_F$ on $\tilde \W_F$. 
The  $\Z$-linear involution of $\tilde {\rm H}_{\Z}(\M_F)$ as defined in \ref{involetal} is denoted by
 $\upiota^F$.
Note that when $F= x_0$ then    $\tilde {\rm H}_{\Z}(\M_F)$ is in fact $\Htg$  and we do not write the exponents $F$.

\medskip

The algebra   $\tilde {\rm H}_\Z(\M_F)$   does not inject in  $\tilde {\rm H}_\Z$  in general. However, there  is a \emph{positive} subring  
 $\tilde {\rm H}_{\Z}(\M_F)_+$   of $\Hh_\Z(\M_F)$ with $\Z$-basis  the set of all $\t_w^F$ for  $w\in \tilde\W_F$ that are $F$-positive,  and an injection 
$$\begin{array}{cccc}j_F^+:&\tilde {\rm H}_{\Z}(\M_F)_+&\longrightarrow& \Htg\cr &\t^F_w&\longmapsto &\t_w\cr\end{array}$$ which,
if $\R$ is a ring containing an inverse for $q.1_R$,  extends to a $\R$-linear injection
$\tilde {\rm H}_{\Z}(\M_F)\otimes _\Z \R\rightarrow \Htg\otimes_\Z \R$  denoted by $j_F$ (the proof in the case of  complex Hecke algebras can be found in \cite[(6.12)]{BK};  it goes through for pro-$p$ Iwahori-Hecke rings over $\Z$).

\section{Representations of spherical and  pro-$p$ Hecke algebras}
\subsection{Hecke algebras attached to parahoric subgroups of $\mathbf G(\Corps)$}
\subsubsection{Parahoric subgroups\label{parahorics}}
Associated to each facet $F$ of  the (semi-simple) building  is, in a $\Gp$-equivariant way, a smooth affine $\mathfrak{O}$-group scheme $\mathbf{G}_F$ whose general fiber is $\mathbf{G}$ and such that $\mathbf{G}_F(\mathfrak{O})$ is the pointwise stabilizer in $\Gp$ of the preimage $\pr^{-1}(F)$  of $F$ in the extended building. %This group scheme is denoted by $\mathscr{G}_{\pr^{-1}(F)}$ in \cite[3.4.1]{Tit}. 
Its connected component is denoted by $\mathbf{G}_F^\circ$ so that the reduction $\overline{\mathbf{G}}_F^\circ$ over $\mathbb{F}_q$ is a connected smooth algebraic group. The subgroup $\mathbf{G}_F^\circ(\mathfrak{O})$ of $\Gp$ is  a parahoric subgroup. 
 We consider
\begin{equation*}
    \I_F := \{ g \in \mathbf{G}_F^\circ(\mathfrak{O}) :( g \mod \varpi) \in\ \textrm{unipotent radical of $\overline{\mathbf{G}}_F^\circ$} \}.
\end{equation*}
The groups $\I_F$ are compact open pro-$p$ subgroups in $\Gp$  such that $\I_C = \I$, $\I_{x_0} = \K_1$ and
\begin{equation}\label{pr1}
    g\I_F g^{-1} = \I_{gF} \qquad\textrm{for any $g \in \Gp$},\:\textrm{ and }\:
    \I_{F'} \subseteq \I_F \qquad\textrm{whenever $F' \subseteq \overline{F}$}.
\end{equation}

 Let  $F$ be  a standard facet. 
Then $\mathbf{G}_F^\circ(\mathfrak{O})$
 is the distinct union of the double cosets ${\I} \hat w {\I}$ for all $w$ in   $\tilde\Wf_F$ \cite[Lemma 4.9 and \S 4.7]{OS}. 
 
 \begin{rem}\label{parahoricgene}

 Denote by $\Uu_F$ the subgroup of $\K$ generated by  all 
 $\Uu_{(\root, -\inf_{F}(\root))}$ for all $\root\in\Phi$.  Then ${\mathbf G}_F^\circ(\mathfrak O) = \Uu_F \Tp^0$ (\cite[5.2.1, 5.2.4]{BTII}).

 \end{rem}

Since $F$ is standard, the product map 
\begin{equation}\label{rootsubgroupeq}
    \prod_{\root\in \Phi^-} \Uu_{(\root, 1)}\times  \Tp^1 \times \prod_{\root\in \Phi_F^+} \Uu_{(\root, 1)}\times  \prod_ {\root\in \Phi^+-\Phi_F^+} \Uu_{(\root, 0)}\overset{\sim}\longrightarrow \I_F
\end{equation}induces a bijection, where the products on the left hand side are ordered in some arbitrary chosen way (\cite[Proposition I.2.2]{SS}).
Denote by $\Uu_{F}^+$ the subgroup of $\I_F$ generated by all
$\Uu_{(\root, 0)}$ for  $\root$ in $ \Phi^+-\Phi_F^+$. Then $\I_F$ is generated by $\K_1$  and $\Uu_F^+$.

\medskip

Let $\Dd_F$ denote the set of elements in $\W$ such that $d^{-1}\Phi_F^+\subseteq \Phi_{aff}^+$. In particular, $\Dd_{x_0}$ coincides with $\Dd$ (\S\ref{rootG}) and contains $\Dd_F$ for any standard facet $F$. 

We have the following (\cite[Remark 4.17 and Proposition 4.13]{OS}): 

\begin{lemma} \begin{enumerate}
\item The set  of all $\hat{\tilde d}$ for $d\in \Dd_F$ is a system of representatives of the double cosets $\mathbf{G}_F^\circ(\mathfrak{O})\backslash \Gp /\I$.
\item For $d\in \Dd_F$, we have  $\I_{F} (d\I d^{-1} \cap {\mathbf{G}}_F^\circ(\Oo)) = \I$.
%and  $d\I d^{-1}\cap  \overline{\mathbf{G}}_F^\circ(\Oo) =d \I d^{-1} \cap \K $.

\end{enumerate}
\label{413}
\end{lemma}

\begin{rem} - The intersection of $\Dd_F$ with $\W_F$ is the distinguished set of representatives  of $\Wf_F\backslash\W_F$ (see \ref{disting}).\label{rema:DF} \\ - The set  of all $\hat{\tilde d}$ for $d\in \Wf\cap \Dd_F$ is a system of representatives of the double cosets $\mathbf{G}_F^\circ(\mathfrak{O})\backslash \K /\I$.
\end{rem}
  \subsubsection{Hecke algebras}
The universal representation $\XX$  for $\Gp$ was defined in \ref{lasko}. Recall that it is a left module for the pro-$p$ Iwahori-Hecke $k$-algebra $\Hh$ which  is isomorphic to $$\Htg\otimes_\Z k$$ where $\Htg$ is the Hecke ring described in \ref{geneHecke}. Remark that the results of \ref{subsec:fourier} apply.
\medskip

 For  $w\in \tilde \W$ (resp. $g\in\Gp$)  we still denote by $\t_w$ (resp. $\t_g$) its image  in $\Hh$.
 Let  $F$ be a standard facet.
  Extending functions on  ${\mathbf G}_F^\circ(\mathfrak O)$ by zero to $\Gp$ induces a ${\mathbf G}_F^\circ(\mathfrak O)$-equivariant embedding
\begin{equation*}
    \mathbf{X}_F := \ind^{{\mathbf G}_F^\circ(\mathfrak O)}_\I(1) \hookrightarrow \mathbf{X} \ .
\end{equation*}  The $k$-algebra
\begin{equation*}
 \H_F := \EInd_{k[{\mathbf G}_F^\circ(\mathfrak O)]}(\mathbf{X}_F) \cong [\ind^{{\mathbf G}_F^\circ(\mathfrak O)}_\I(1)]^\I \ .
\end{equation*}
is naturally a subalgebra of $\Hh$ via the extension by zero  embedding $[\ind^{{\mathbf G}_F^\circ(\mathfrak O)}_\I(1)]^\I \hookrightarrow \ind^{\Gp}_\I(1)$.  

\begin{pro}\begin{enumerate}\item The finite Hecke algebra $\H_F$  has basis the set of all $\t_w$ for $w\in \tilde \Wf_F$. 
\item It is a Frobenius algebra over $k$. In particular, for any  (left or right) $\H_F$-module $\m$, we have an isomorphism of vector spaces $\Hom_{\H_F}(\m, \H_F)= \Hom_k(\m, k).$
\item The Hecke algebra $\Hh$ is a free  left $\H_F$-module with basis the set of all $\t_{\tilde d}$, $d\in \Dd_F$. 

\end{enumerate}
\label{libertepro} 
\end{pro}

{
\begin{proof} 
The first point is clear.
For iii, see \cite[Proposition 4.19]{OS}. 
For ii. see \cite[Thm 2.4]{Sawada} (or  \cite[Proposition 6.11]{CE}). Recall that ${\H_F}$ being Frobenius  means that it is finite
dimensional over $k$ and that it carries a $k$-linear form
$\delta$ such that the bilinear form $(a,b) \longmapsto
\delta(ab)$ is nondegenerate.  In particular,
there is a unique map $\iota : {\H_F} \longrightarrow {\H_F}$ satisfying
 $ \delta(\iota(a)b) = \delta(ba)$ for any $a,b \in {\H_F}$.
One shows that $\iota$ is an automorphism of the $k$-algebra ${\H_F}$. For any left or right ${\H_F}$-module $\m$ we let $\iota^\ast \m$, resp.\ $\iota_\ast \m$, denote $\m$ with the new ${\H_F}$-action through the automorphism $\iota$, resp.\ $\iota^{-1}$.
Then for any left, resp.\ right,  ${\H_F}$-module $\m$ it is classical to establish that the map $f\mapsto \delta\circ f$ 
   is an isomorphism of right, resp.\ left,  ${\H_F}$-modules between $\Hom_{\H_F}(\m,{\H_F})$  and $\Hom_k(\iota^\ast \m,k)$ (resp. $\Hom_k(\iota_\ast \m,k)$).
%by using a free presentation of
%$\m$ we are reduced to the case $\m = {\H_F}$; but in this case the map in
%question is the duality isomorphism arising from the nondegeneracy
%of the bilinear form $\delta$.

\end{proof}

\begin{rem}\label{ParabParah} The previous definitions and results are valid when replacing
$\Gp$ by a semi-standard Levi subgroup.  
We will denote by $\Hh(\M_F)$ the pro-$p$ Iwahori-Hecke algebra of $\M_F$ with coefficients in $k$. It is isomorphic to $\Hh_\Z(\M_F)\otimes_\Z k$.

As for the finite dimensional Hecke algebras associated to parahoric subgroups of $\M_F$, we will only consider the following situation. Let $F$ be a standard facet and $\M_F$ the associated Levi-subgroup.
By \cite[Lemma 2.9.1]{HainesBC}, $\M_F\cap \K$  is the parahoric subgroup of $\M_F$ corresponding to an 
hyperspecial point $ x_F$ of the building of $\M_F$. The corresponding finite Hecke algebra $\H_{x_F}(\M_F)$ has basis the set of all $\t_{w}^F$ for $w\in \tilde \Wf_F$. 
\end{rem}
\subsubsection{\label{characters}} When $F= x_0$, we  write $\H$ instead of $\H_{x_0}$. 
Consider  a simple $\H$-module. By \cite[(2.11)]{Sawada0} %(see also \cite[Theorem 6.10]{CE})),  
it is one dimensional and we denote by $\chi:\H\rightarrow k$ the  corresponding character. Let  $\bar\chi$ be the character of $\overline{\mathbf T}(\fq)$ given by 
$$\bar\chi(t):=\chi(\t_t)$$ and $\epsilon_{\bar \chi}$ the corresponding idempotent  (\S\ref{subsec:fourier}).  We  have $\chi(\epsilon_{\bar\chi})=1$.
Let $\Pi_{\bar\chi}$ denote the set of simple roots $\root\in\Pi$ such that $s_\root\bar\chi=\bar \chi$. 
For $\root\in \Pi$, we have (by the quadratic relations \eqref{Q'}):
 $\chi(\t_{n_\root})=0$ if $\root\in \Pi-\Pi_{\bar\chi}$  and
$\chi(\t_{n_\root})\in \{0, -1\}$ otherwise. Define
$\Pi_\chi$ to be the set of all $\root\in\Pi_{\bar\chi}$ such that  $\chi(\t_{n_\root})=0$.

\medskip

%\green{Denote by $\hat {\overline {\mathbf T}}(\fq)$ the group of $k$-characters of  $\overline {\mathbf T}(\fq)$.}

\noindent A $k$-character $\chi$ of $\H$ is parameterized by the following data:\\
- a $k$-character $\bar\chi$ of $\overline{\mathbf T}(\fq)$ and the attached $\Pi_{\bar\chi}$ as above.\\
- a subset $\Pi_\chi$ of $\Pi_{\bar\chi}$ such that  for all $\root\in \Pi$,  we have $\chi(\t_{n_\root})=-1$ if and only if $\root\in \Pi_{\bar\chi}-\Pi_{\chi}$.\\

\subsubsection{\label{subsec:bcn}}Let $(\rho, \V)$ be a weight. Denote by ${\mathcal H}(\Gp, \rho )$  the algebra of $\Gp$-endomorphisms of the compact induction $\ind_\K ^\Gp\rho $.
Choose  and fix $v$ a basis for $\rho^\I$ (it is known that this space is one dimensional \cite[Corollary 6.5]{CL},   see also Theorem \ref{theoCE} below which is drawn from \cite{Cabanes}). Denote by $ \1_{\K, v}$
the function of $\ind_{\K }^\Gp \rho $ with support $\K $ and value $v$ at $1$. It is $\I$-invariant. Since $(\rho ,\V)$ is irreducible, an element  $T$ of  $\cal H (\Gp, \rho )$ is  determined by the image $T(\1_{\K , v})$ of $ \1_{\K , v}$.  
 The restriction  to $(\ind_{\K }^\Gp\rho)^\I$ therefore yields 
an injective morphism  of $k$-algebras
\begin{equation}\label{bcn}\cal H (\Gp, \rho ) \longrightarrow 
\Hom_ {\Hh }(\ind_{\K }^\Gp\rho)^\I, (\ind_{\K }^\Gp\rho)^\I)\: \end{equation}
In \S\ref{prelim}, we will prove that this is an isomorphism.
We first identify the structure of  $\Hh$-module $(\ind_{\K }^\Gp\rho)^\I$.

\begin{lemma} \label{apple}

 We have  a $\Hh$-equivariant isomorphism given by
\begin{equation}\label{isofond0}\begin{array}{ccc}\chi\otimes _\H \Hh&\cong& (\ind_{\K }^\Gp\rho )^{\I}\cr 1\otimes 1&\mapsto& \1_{\K , v}\cr   \end{array}\end{equation}

\end{lemma}

\begin{proof} We make the following remark: for $d\in \Dd$,  the action of  $\t_{\tilde d}$  on the right on  $\1_{\K , v}$  gives the unique $\I$-invariant element of $\ind_{\K }^\Gp\rho $ with support in $\K \hat {\tilde d}\: \I$ and value $v$ at $\hat {\tilde d} $;  the set of all  such elements when $d$ ranges over $\Dd$  is  a basis for $(\ind_{\K }^\Gp\rho )^{\I}$.

%\item The right $\HhI$-modules  $\chi\otimes _\H \HhI\cong (\ind_{\K }^\Gp\rho )^{\I}$  are isomorphic via the map
% $1\otimes 1\mapsto 1_{\K, v}$. 

The key observation  to verify this claim is that, by  Lemma \ref{413}, we have  $\K_1(\tilde d \I \tilde d^{-1}\cap\K)=\I$. The first point of the remark follows easily.
Furthermore,  an $\I$-invariant function $f\in  (\ind_{\K }^\Gp\rho )^{\I}$ is determined by its values at all $\hat {\tilde d}$'s for $ d\in   \Dd$ which  are  $\tilde  d\I \tilde d^{-1}\cap \K$-invariant vectors of $\V$: these vectors are $\I$-invariants and therefore  proportional to $v$. It proves the second point. 

The surjectivity of the map \eqref{isofond0}  follows easily. The injectivity comes from the fact that  a basis for $\chi\otimes _\H \Hh$ is given by all $1\otimes \t_{\tilde d}$ for $d\in\Dd$ (Proposition \ref{libertepro}).
\end{proof}

\begin{rem}\label{support}Recall that an element of $\cal H (\Gp, \rho )$ can   be seen as a function  with compact support  $f:\G\rightarrow \EInd_{\k}(\V)$ such that $f(kgk')=\rho (\kappa) f(g) \rho  (\kappa')$ for any $g\in \Gp$, $\kappa,\kappa'\in \K $.
To such a function $f$ corresponds the Hecke operator  $T_f\in \cal H (\Gp, \rho )$ that sends $\1_{\K , v}$ on
the element of $\ind_{\K }^\Gp\rho  $ defined by $g\mapsto f(g)\, . \,v$. 
Reciprocally, to an element $T\in \cal H (\Gp, \rho ) $ is  associated the function
$f_T:\G\rightarrow \EInd_{k}(\V)$ defined by $f_T(g): w\mapsto T(\1_{\K , w}) [g]$  for any $g\in \Gp$, where
$\1_{\K , w}\in\ind_{\K}^\Gp \rho$ is the unique function with support $\K$ and value $w\in \V$ at $1$.
For  $\lambda\in \X_*(\T)$, the function $f_T$ has support in $\K \lambda(\varpi)\K$ if and only if
$T(\1_{\K , v})\in \ind_{\K}^\Gp \rho$ has support in $\K\lambda(\varpi)\K$.

\end{rem}

\medskip

\subsection{Categories of Hecke modules and  of representations of parahoric subgroups\label{categories}}
Let $F$ be a standard facet. 
We consider the abelian  category of  (smooth) representations of ${\mathbf G}_F^\circ(\mathfrak O)$.
Define the functor $\dagger$   that associates to  a smooth representation $\V$  of ${\mathbf G}_F^\circ(\mathfrak O)$
the subrepresentation $\V^\dagger$  generated by $\V^\I$.
Consider the  following  categories of representations:

\begin{enumerate}
\item[a)]  $\Rr(F)$  is the category of   finite dimensional    representations of ${\mathbf G}_F^\circ(\mathfrak O)$ with
trivial action of the normal subgroup $\I_F$. It  is equivalent to the (abelian) category of finite dimensional representations of the finite  reductive   group  $\overline{\mathbf G}_F ^{red}(\mathbb F_q)={\mathbf G}_F^\circ(\mathfrak O)/\I_F$ (see \cite[Proof of Lemma 4.9]{OS}).  The  irreducible representations of ${\mathbf G}_F^\circ(\mathfrak O)$ are the simple objects  in $\Rr(F)$.\\

Note that  $\dagger$  induces a functor $\dagger: \Rr(F)\rightarrow \Rr(F)$. 
The category $\Rr(F)$ is also  equipped with the  endofunctor $ \vee:\: \V\mapsto \V^\vee $  
associating to $\V$ the contragredient representation  $\V^\vee=\Hom_k(\V, k)$. 
Since $\vee$ is  anti-involutive,   $\V^\vee$ is irreducible if and only if $\V$ is irreducible.
\\
\item[b)]  $\Rr^\dagger(F)$ is the full subcategory of  $\Rr(F)$ image of the functor $\dagger$. 
Any irreducible representation of ${\mathbf G}_F^\circ(\mathfrak O)$ is an object  in $\Rr^\dagger(F)$.
By adjunction,
a representation $\V\in  \Rr(F)$ is an object of  $\Rr^\dagger(F)$ if and only if $\V$ sits in an exact sequence in $\Rr(F)$ of
the form\begin{equation}\label{quotient}\XX_F^\ell\rightarrow \V\rightarrow 0\end{equation} for some $\ell\in \N$, $\ell\geq 1$.\\

\item[c)]  $\BB(F)$   is the  full (additive) subcategory of  $\Rr^\dagger (F)$ whose objects are the $\V\in \Rr^\dagger(F)$ such that
$\V^\vee$ is also an object in $\Rr^\dagger(F)$. 
Any irreducible representation of ${\mathbf G}_F^\circ(\mathfrak O)$ is an object  in $\BB(F)$. By the following proposition, this definition coincides with \cite[Definition 1]{Cabanes}.
\end{enumerate}

\begin{pro}\phantomsection
\begin{itemize}
\item[i.]  In $\Rr(F)$, we have  $\XX_F\cong \XX_F^\vee$.
\item[ii.] A representation $\V\in  \Rr(F)$ is an object of  $\mathscr B(F)$ if and only if 
there are  $ \ell, m\geq 1$ and $f\in \Hom_{\Rr(F)}(\XX_F^m, \XX_F^\ell)$ such that $\V= \Im(f)$.

\end{itemize}
 \label{contra}

\end{pro}

\begin{proof} i. Let $\phi:\XX_F\rightarrow \XX_F^\vee$ be the unique  ${\mathbf G}_F^\circ(\mathfrak O)$-equivariant map sending
the characteristic function of $\I$ onto $ \XX_F\rightarrow k, f\mapsto f(1)$. One easily checks that it is  well-defined, injective, and therefore surjective.
ii. Let  $\V\in  \Rr(F)$. We deduce the claim from  by i. by observing  that,  $\V^\vee\in \Rr^\dagger(F)$ if and only if  $\V$ sits in an exact sequence 
 in $\Rr(F)$ of
the form\begin{equation}\label{sousmod}0\rightarrow \V\rightarrow \XX_F^\ell\end{equation} for some $\ell\in \N$, $\ell\geq 1$.

\end{proof}

\begin{rem}\label{IrrInB} An irreducible representation $\V$  of ${\mathbf G}_F^\circ(\mathfrak O)$ is an  object  in $\BB(F)$. The work of Carter and Lusztig \cite{CL} describes $\V$  explicitly  as the image of a ${\mathbf G}_F^\circ(\mathfrak O)$-equivariant morphism $\XX_F\rightarrow \XX_F$.

\end{rem}
Consider the category $\Mod(\H_F)$ of finite dimensional modules over $\H_F$. The functor 
\begin{equation}\Rr^\dagger (F)\rightarrow \Mod(\H_F), \: \V\mapsto \V^\I\label{func}\end{equation} is faithful.
The following theorem is  \cite[Theorem  2]{Cabanes} the proof of which  relies on the fact that $\H_F$ is self-injective (see Proposition \ref{libertepro}). 
%Note that the results of \cite{CE} apply because the finite reductive group $\overline{\mathbf G}_F ^{red}(\mathbb F_q)$ is equipped  with a strongly  split $BN$ pair of characteristic $p$ (\cite[Lemma 5.2]{HV}).

\begin{theorem}\label{theoCE}
The functor \eqref{func} induces an equivalence between $\BB(F)$ and $\Mod(\H_F)$.\end{theorem}

\begin{rem}
 In particular,
\eqref{func} is faithful and essentially surjective. It is not full in  general (see \cite{RV}). 
\end{rem}

For  $ \V$ in $\Rr(F)$ we consider the (compactly) induced representation 
$\ind{\,}_{{\mathbf G}_F^\circ(\mathfrak O)}^\K (\V)$.

\begin{lemma}
\begin{enumerate}
\item  $\ind{\,}_{{\mathbf G}_F^\circ(\mathfrak O)}^\K (\V)$ is a representation in $ \Rr(x_0)$.
\item   We  have $(\ind{\,}_{{\mathbf G}_F^\circ(\mathfrak O)}^\K (\V))^\dagger= \ind{\,}_{{\mathbf G}_F^\circ(\mathfrak O)}^\K (\V^\dagger)$ in $\Rr(x_0)$.
\end{enumerate}
 \label{IndparDagger} 
%The  $\H$-modules $(\ind{\,}_{{\mathbf G}_F^\circ(\mathfrak O)}^\K (\Pi))^{\I}$ and $\Pi^\I\otimes _{\H_F}\H$ are isomorphic.
\end{lemma}
\begin{proof}  It is clear that both $\ind{\,}_{{\mathbf G}_F^\circ(\mathfrak O)}^\K (\V)$ and $\ind{\,}_{{\mathbf G}_F^\circ(\mathfrak O)}^\K (\V^\dagger)$ are in $\Rr(x_0)$ because $\K_1$ is normal in  $\K$ and contained in $\I_F$. Furthermore
$\ind{\,}_{{\mathbf G}_F^\circ(\mathfrak O)}^\K (\V^\dagger)$  is an object in $\Rr^\dagger(x_0)$: it is generated as a representation of $\K$ by the functions with support in  ${\mathbf G}_F^\circ(\mathfrak O)$ taking value in $\V^\I$ at $1$.
It remains to show that  the natural injective morphism of representations of $\K$
\begin{equation}\label{identification}\ind{\,}_{{\mathbf G}_F^\circ(\mathfrak O)}^\K (\V^\dagger)\rightarrow  (\ind{\,}_{{\mathbf G}_F^\circ(\mathfrak O)}^\K (\V))^\dagger\end{equation} is surjective: by Mackey decomposition, an $\I$-invariant function $f\in  \ind{\,}_{{\mathbf G}_F^\circ(\mathfrak O)}^\K (\V)$ is completely determined by its values at all $\kappa$ in a system of representatives of
the double cosets ${\mathbf G}_F^\circ(\mathfrak O)\backslash \K/\I$ and the value of $f$ at $\kappa$ can be any element in 
$\V^{{\mathbf G}_F^\circ(\mathfrak O)\cap \kappa \I \kappa^{-1}}= \V^{\lp \I_F, {\mathbf G}_F^\circ(\mathfrak O)\cap \kappa \I \kappa^{-1}\rp}$. Choose the system of representatives given by Remark \ref{rema:DF}ii. Then by Lemma \ref{413}ii,  the value of $f$  at $\kappa$ can be any value in $\V^\I$ and $f$ lies in the image of \eqref{identification}.

\end{proof}

\subsection{\label{prelim}Spherical Hecke algebra attached to a weight}

  \medskip
  
Let $(\rho , \V)$ be a weight and $\chi: \H\rightarrow k$ the corresponding character.  By Cartan decomposition (\S\ref{cartan}), the compact induction  $\ind_{\K }^ {\Gp}\rho$ decomposes as a $k[[\K]]$-module  into the direct sum $$\ind_{\K }^ {\Gp}\rho= \bigoplus _{\lambda\in \X_*^+(\T)} \ind_{\K }^ {\K \lambda(\varpi)\K}\rho$$   of
the spaces of functions with support in  $\K\lambda(\varpi)\K$. 
 The following proposition is proved  after the subsequent corollary which is the main result of this section:  it allows us to replace the study of  the 
spherical algebra $\cal H (\Gp, \rho )$ by the one of  $\Hom_\Hh(\chi\otimes_\H\Hh,\chi\otimes_\H\Hh )$ 
which is achieved in Section \ref{main} (see Proposition \ref{prop:main}).

\begin{pro} \label{categoryB}
Let  ${\lambda\in \X_*(\T)}$.\begin{enumerate}
\item  The representation 
$(\ind_{\K }^ {\K \lambda(\varpi)\K}\rho)^\dagger$ of $\K$  lies in $\mathscr B(x_0)$.
\item  The space  $\Hom_\H(\chi,  (\ind_{\K }^ {\K \lambda(\varpi)\K}\rho)^\dagger)$ is at most one dimensional.
\end{enumerate}

%\item As a $k[[\K]]$-module, the sub-$\Gp$-representation  of
%$\ind_{\K }^\Gp \rho$  generated by  the subspace $(\ind_{\K }^\Gp \rho)^\I$
%decomposes into the direct sum $ \bigoplus _{\lambda\in \X_*^-(\T)} (\ind_{\K }^ {\K \lambda(\varpi)\K}\rho)^\dagger$  of
%the functions with support in  $\K\lambda(\varpi)\K$.

\end{pro}

\begin{cor}

\begin{enumerate}
\item  The map \eqref{bcn} induces an isomorphism of $k$-algebras 
\begin{equation}\label{bcn1}\cal H (\Gp, \rho )= \Hom_\Hh((\ind_{\K }^\Gp \rho)^\I,(\ind_{\K }^\Gp \rho)^\I )=\Hom_\Hh(\chi\otimes_\H\Hh,\chi\otimes_\H\Hh ).\end{equation}
\item For ${\lambda\in \X_*(\T)}$, the subspace of $\cal H (\Gp, \rho )$ of the functions with support in $\K \lambda(\varpi) \K$ is at most one dimensional.

\end{enumerate}

\label{CoroIso}
\end{cor}
\begin{rem}\label{dim1} It will be a corollary of the proof of  Proposition \ref{prop:main} that the subspace of $\cal H (\Gp, \rho )$ of the functions with support in $\K \lambda(\varpi) \K$ is in fact one dimensional. This fact is  proved   and used in \cite{satake}  (Step 1 of proof of Theorem 1.2)  but our method is independent.

\end{rem}

\begin{proof}[Proof of the Corollary] By adjunction, we have 
$$ \cal H (\Gp, \rho )=
 \Hom_\K(\rho, \ind_{\K }^\Gp \rho)=  \Hom_\K(\rho, \ind_{\K }^\Gp \rho)=  \oplus _{\lambda\in \X_*^+(\T)} \Hom_\K(\rho, \ind_{\K }^ {\K \lambda(\varpi)\K}\rho)$$ and by Proposition \ref{categoryB}i and Theorem \ref{theoCE}
 \begin{align*} \cal H (\Gp, \rho )= & \oplus _{\lambda\in \X_*^+(\T)}
 \Hom_\K(\rho,    (\ind_{\K }^ {\K \lambda(\varpi)\K})^\dagger)\cr=&\oplus _{\lambda\in \X_*^+(\T)}\Hom_\H(\chi, (\ind_{\K }^ {\K \lambda(\varpi)\K})^\I)= \Hom_\H(\chi, (\ind_{\K }^\Gp \rho)^\I).\cr=&\Hom_\Hh(\chi\otimes _\H \Hh, (\ind_{\K }^\Gp \rho)^\I)= \Hom_\Hh((\ind_{\K }^\Gp \rho)^\I, (\ind_{\K }^\Gp \rho)^\I)\cr\end{align*} 
where the last equality comes from  Lemma \ref{apple}. The second statement of the corollary then comes from the 
one of the proposition using Remark \ref{support}.

\end{proof}

\def\lap{\lambda(\varpi)}

\begin{proof}[Proof of Proposition \ref{categoryB}]  It suffices to show the proposition  for $\lambda\in \X_*^+(\T)$.
% In this proof we denote  by $\lap$ the element $\lambda(\varpi)\in \Gp$.
We  first describe  the $\K_1$-invariant subspace of $\ind_{\K}^{\K  \lambda (\varpi)\K}\rho$ because it contains
$(\ind_{\K }^ {\K \lap\K}\rho)^\dagger$. Set $$\K_\lambda:= \K\cap \lap^{-1}\K \lap.$$ As a $k[[\K]]$-module, $\ind_{\K}^{\K \lap \K}\rho$ is isomorphic  to the compact induction
$\ind_{\K_\lambda}^\K\lambda_\ast(\rho)$
where $\lambda_\ast(\rho)$ denotes the space $\V$ with the group $\K_\lambda$ acting through the homomorphism 
$ \K_\lambda \xrightarrow{\; \lap \,.\, \lap^{-1}\ \;} \K$. 

\medskip

 Since $\K_1$ is normal in $\K$,  we have the representation   $(\lambda_\ast(\rho))^{\K_\lambda\cap \K_1}$ of
$\K_\lambda$ on the space $$\V_\lambda:=\V^{\K\cap\lap \K_1\lap^{-1}}=\V^{\lp\K\cap\lap \K_1\lap^{-1}, \K_1\rp}.$$ It  can be extended to a  representation  $(\uppi_\lambda, \V_\lambda)$ of $\P_\lambda:=\K_\lambda\K_1$ that factors through
$\P_\lambda/\K_1\simeq \K_\lambda/\K_\lambda\cap \K_1$.

%$$(\ind_{\K_\lambda}^\K\lambda_\ast(\rho))^{\K_1}= \ind_{\Pp_\lambda }^\K   (\lambda_\ast(\rho))^{\K_\lambda\cap \K_1}= \ind_{\lap \K_\lambda, \K_1\rp}^\K   \lambda_\ast(\rho^{ \K\cap \lap \K _1\lap^{-1}}).$$

\medskip

Denote by $\Wf_\lambda$  the stabilizer of $\lambda$ in $\Wf$. Since $\lambda\in \X_*^+(\T)$, 
it  is generated by  the simple reflections $s_\root$ for all $\root\in \Phi$ such that $\lp\lambda, \root\rp=0$. Denote by $F_\lambda$ the associated standard  facet.
The attached subset $\Phi_{F_\lambda}$  of $\Phi$ consists in all the roots $\root$ 
such that $\lp\lambda, \root\rp=0$. The closure of $F_\lambda$  is the set of points in $x\in \overline C$ such that
$ \root(x)=0$ for all 
$\root\in \Phi_{F_\lambda}$.

\begin{fact} \label{f1}The subgroup $\P_\lambda$ of $\K$ is the  parahoric subgroup  associated to $\F_\lambda$. 
\end{fact}

%Note that for $\root\in\Phi^+$ and $x\in F_\lambda$, we have $0\leq \root(x)\leq 1$ (explain).

\begin{fact} \label{f2}As $k[[\K]]$-modules, we have
$(\ind_{\K_\lambda}^\K\lambda_\ast(\rho))^{\K_1}= \ind_{\P_\lambda}^{\K}   \uppi_\lambda.$
\end{fact}

\begin{fact} \label{f3} 
%The  $\I$-invariant subspace of $\uppi_\lambda$ is $1$-dimensional. 
We have  a) $\uppi_\lambda\in \Rr(F_\lambda)$, b) $\uppi_\lambda^\dagger$ is irreducible,
c) $\Hom_{\H_{F_\lambda}}(\chi, \uppi_\lambda^\I)$ is at most one dimensional.

\end{fact}

We deduce from Fact \ref{f2} that  $(\ind_{\K }^ {\K \lambda(\varpi)\K}\rho)^\dagger= (\ind_{\P_\lambda}^{\K}   \uppi_\lambda)^\dagger$ so that, to prove the proposition,  it remains to check that   $(\ind_{\P_\lambda}^{\K}   \uppi_\lambda)^\dagger$ is a representation in $\mathscr{B}(x_0)$.
 By Fact \ref{f3} b) and  Remark \ref{IrrInB},  there is an injective $\P_\lambda$-equivariant map $\uppi_\lambda^\dagger\rightarrow {\XX_{F_\lambda}}$
which, by exactness of compact induction, gives an  injective $\K$-equivariant map 
\begin{equation}\ind_{\P_\lambda}^\K \uppi_\lambda^\dagger\rightarrow \ind_{\P_\lambda}^\K{\XX_{F_\lambda}}=\XX_{x_0}.\label{injX}\end{equation}
Since furthermore, by Fact \ref{f3} a) and Lemma \ref{IndparDagger}, we have $(\ind_{\P_\lambda}^\K \uppi_\lambda)^\dagger=
\ind_{\P_\lambda}^\K (\uppi_\lambda ^\dagger)$, we just proved that  the $\K$-representation 
 $(\ind_{\P_\lambda}^\K \uppi_\lambda)^\dagger$ injects in $\XX_{x_0}$. By Proposition \ref{contra}ii,   the representation $(\ind_{\P_\lambda}^\K \uppi_\lambda)^\dagger$ is  therefore an object in  $\mathscr B(x_0)$. It is the first statement of the proposition.

\medskip

 In passing, we deduce from \eqref{injX} that there is a right $\H$-equivariant injection
$$(\ind_{\K }^ {\K \lambda(\varpi)\K}\rho)^\I \longrightarrow \H$$ so that
 $\Hom_\H(\chi,  (\ind_{\K }^ {\K \lambda(\varpi)\K}\rho)^\I)$  injects in $\Hom_\H(\chi, \H)$ which is one dimensional by  Proposition \ref{libertepro}ii.  It gives the second statement of the proposition.

 \medskip

We now prove the  Facts.  Recall that $\lap\in \Tp$ is a lift for $e^{-\lambda}\in \W$ (Remark \ref{rema:normalization}) and that for all $\root\in \Phi$, we have
$\lap  \Uu _{(\root, 0)} \lap^{-1}= \Uu _{e^{-\lambda}(\root, 0)}=
\Uu _{(\root, \lp\lambda, \root\rp)}$.

\medskip

\noindent\emph{Proof of Fact \ref{f1}:}
From \eqref{rootsubgroupeq}  we deduce that
the subgroup $\Uu_C ^+$   of $\I$ generated by  all the
root subgroups $\Uu _{(\root,0)}$ for $\root\in \Phi^+$
 is contained in $\K_\lambda$:  indeed,
let $\root\in \Phi^+$, we have   $\lp \lambda, \root\rp \geq 0$ and 
$\lap  \Uu _{(\root, 0)} \lap^{-1}\subset \I$, therefore  $ \Uu _{(\root, 0)} \subseteq \K\cap\lap^{-1} \K \lap $.  \\
Now recall that $\P_\lambda= \lp\K_\lambda, \K_1\rp$. 
The pro-$p$ Iwahori subgroup $\I$ which is generated by $\K_1$ and $\Uu_C ^+$ is    contained in $\P_\lambda$, and so is $\Iw$ since  $\T^0\subseteq \P_\lambda$ (\cite[4.6.4 ii]{BTII}). We have proved that $\P_\lambda$ is the parahoric subgroup corresponding to a standard facet  (This statement is in fact enough for the proof of the proposition).
It remains to check that it is equal to  ${\mathbf G}_{F_\lambda}^\circ(\mathfrak O)$ which, by Remark \ref{parahoricgene} is the subgroup of 
$\K$ generated by $\T^0$,  all $\Uu_{(\root, 0)}$ for  $\root\in\Phi$ 
such that $\lp\lambda, \root \rp \geq 0$ and
all
$\Uu_{(\root, 1)}$ for  $\root\in\Phi$  such that  $\lp\lambda, \root\rp < 0$.
But $\lap \Uu_{(\root, 0)}\lap^{-1}\subset \K$ if and only if $\lp\lambda, \root\rp\geq 0$. It proves the required equality (using \eqref{rootsubgroupeq} for $\P_\lambda$).

\medskip

\noindent\emph{Proof of Fact \ref{f2}:}   Since $\K_1$ is normal in $\K$,  a ${\K_1}$-invariant function $f$ in  
$\ind_{\K_\lambda}^\K\lambda_\ast(\rho)$ is entirely determined by its values at  the points of a  system of representatives of the  cosets $\K_\lambda\backslash \K/\K_1=\P_\lambda\backslash \K$ and these values can be any elements in $\V_\lambda$.
The   $\P_\lambda$-equivariant map 
$$\uppi_\lambda\rightarrow (\ind_{\K_\lambda}^\K\lambda_\ast(\rho))^{\K_1}$$
carrying an element $v\in \V_\lambda$ to the unique $\K_1$-invariant funtion $f\in  \ind_{\K_\lambda}^\K\lambda_\ast(\rho)$ with value $v$ at $1$ therefore induces the expected isomorphism of $\K$-representations.

\medskip
\noindent\emph{Proof of Fact \ref{f3}:}  
a) We want to prove that  the pro-unipotent radical $\I_{F_\lambda}$ of $\P_\lambda$ acts trivially on $\uppi_\lambda$.
By \eqref{rootsubgroupeq}, it is generated by  $\K_1$ and the root subgroups
$\Uu _{(\root, 0)}$ for  $\root\in \Phi^+-\Phi_{F_\lambda}^+$ that is to say  for  $\root\in \Phi^+$ satisfying
$\lp\lambda, \root\rp>0$.
Since $\K_1$ acts trivially on $\uppi_\lambda$, we  only need to check that for   $\root\in \Phi^+$ with $\lp\lambda, \root\rp>0$, the action of
$\lap\Uu _{(\root,0)}\lap^{-1}$ on $\V_\lambda$   via $\rho$ is trivial, but it is clear because this group is contained in $\K_1$.
  \\b) Since \eqref{func} is  faithful,   proving that  $\uppi_\lambda ^{\I}$ has dimension $1$  is enough to prove that $\uppi_\lambda^\dagger$  is an irreducible representation of $\P_\lambda$.
We have $\{0\}\neq \uppi_\lambda ^{\I}= \uppi_\lambda ^{\Uu_C ^+}$ and $\Uu_C ^+\subset \K_\lambda$ so the space of  $\uppi_\lambda ^{\I}$ is $\V^{\lp  \K\cap \lap \K _1\lap^{-1},  \lap\, {\Uu _C^+} \lap^{-1}\rp}.$ Let $w$ be the element with maximal length  in $\Wf_\lambda$. 
Denote by $\Uu_C^-$  the
 subgroup    of $\K$  generated by  all the
root subgroups $\Uu _{(\root,0)}$ for $\root\in \Phi^-$.
We claim that
\begin{equation}\label{inclusion-}\lp  \K\cap \lap \K _1\lap^{-1},  \lap {\Uu _C^+}\lap^{-1}\rp\supseteq \hat w \, {\Uu _C^-}  \hat w^{-1}.\end{equation}
 Indeed, let $\root\in \Phi^-$ and  recall that $\lp\lambda,\root\rp\leq 0$:\\
- if $\lp \lambda, \root\rp=0$ then $\root\in \Phi_{F_\lambda} ^-$ and $w\root\in \Phi_{F_\lambda}^+$  so 
 $\lap ^{-1}\hat w\Uu _{(\alpha, 0)}\hat w \lap=
 \Uu_{ (w\root, -\lp \lambda, w\root\rp)}=\Uu_{(w\root, 0)}$ is contained in $\Uu_C^+$; \\
-  if   $\lp \lambda, \root\rp<0$ then
 $\lap ^{-1}\hat w\Uu _{(\alpha, 0)}\hat w \lap=
\Uu_{ (w\root, -\lp \lambda, w\root\rp)}=\Uu_{(w\root, -\lp \lambda, \root\rp)}$ is contained in $\K_1$.

\medskip

We deduce from \eqref{inclusion-} that  $$\V^{\lp  \K\cap \lap \K _1\lap^{-1},  \lap\, {\Uu _C^+} \lap^{-1}\rp}\subseteq
 \V^{\lp  \hat w \, {\Uu_C ^-}  \hat w^{-1}, \,\K_1\rp}$$ and the last space has dimension $1$ because 
$\lp  \hat w \, {\Uu_C ^-}  \hat w^{-1}, \,\K_1\rp$ is a $\K$-conjugate of $\I$ (it is the pro-unipotent radical of the parahoric subgroup attached to the facet $\hat w\hat w_0 C$ where $w_0$ denotes the longest element in $\Wf$).

 \end{proof}

\subsection{Parameterization of the weights\label{param}}
Recall that a weight is an irreducible representation of $\K$ that
is to say  a simple object in $\Rr(x_0)$. By  \cite[Corollary 7.5]{CL} (and also Theorem \ref{theoCE}), the weights are in one-to-one correspondence with the characters  of  $\H$. A character $\chi: \H\rightarrow k$ is  determined by the data of    a morphism $\bar\chi: \T^0/\T^1\rightarrow k^\times$  such that $\bar\chi(t)=\chi(\t_t)$ for all $t\in \T^0/\T^1$, and of
the  subset $\Pi_\chi$ of $\Pi_{\bar\chi}$ (notations in \ref{characters}) defined by:
$$\chi(\t_{n_\root})=-1\textrm{ if and only if }\root\in \Pi_{\bar\chi}-\Pi_{\chi}.$$
To the subset $\Pi_\chi$ of $\chi$ is attached a standard facet $F_\chi$ (Remark \ref{facet}).

\begin{rem}\label{stabi}By \cite[Proposition 6.6]{CL}, the stabilizer of $\rho^\I$ in $\K$ is  equal to the parahoric subgroup
 ${\mathbf G}_{F_\chi}^\circ(\mathfrak O)$ with associated finite Weyl group generated by all $s_\root$, $\root\in \Pi_\chi$. We will denote the latter 
  by $\Wf_\chi$.
\end{rem}

\section{Bernstein-type map attached to a  weight and Satake isomorphism\label{main}}
\subsection{Commutative subrings  attached to a standard facet\label{subsec:commu}}
We fix for the whole section \ref{subsec:commu} a standard facet $F$.

\subsubsection{\label{cones}} 

Consider the subset of all $\lambda\in  \X_*(\T)$  such that $ \lp \lambda, \root \rp\geq 0$ for all $\root\in (\Phi^+- \Phi_F^+)\cup \Phi_F^-$. If $w_F$  denotes the element with maximal length in $\mathfrak W_F$, then 
this set is  the  $w_F$-conjugate of  $\X_*^+(\T)$. 
Bearing in mind the conventions introduced in \ref{notations:tame}, we introduce

\begin{center}$\Cute^+(F):=\{\lambda\in \tilde \X_*(\T)\textrm{ such that $ \lp \lambda, \root \rp\geq 0$ for all $\root\in (\Phi^+- \Phi_F^+)\cup \Phi_F^-$}\}$

$\Cute^-(F):=\{\lambda\in \tilde \X_*(\T)\textrm{ such that $ \lp \lambda, \root \rp\geq 0$ for all $\root\in (\Phi^-- \Phi_F^-)\cup \Phi_F^+$}\}.$  \end{center}
\begin{rem}  \label{rema:length}For all $\lambda$, $\lambda'\in  \Cute^+(F)$  (resp. $\Cute^-(F)$) we have 
$\ell(e^\lambda)+\ell(e^{\lambda'})=\ell(e^{\lambda+\lambda'}).$
\end{rem}
\subsubsection{Bernstein-type maps attached to a standard facet}
\begin{pro}\phantomsection
 \label{themaps} 
\begin{itemize}
\item[i.] There is a unique
morphism of  $\Z[q^{\pm 1/2}]$-algebras
\beq\Theta_{F}^+ :\Z[q^{\pm 1/2}][\tilde \X_*(\T)] \longrightarrow { \Htg}\otimes_\Z \Z[q^{\pm 1/2}]\eeq
such that $\Theta^+_{F}(\lambda)=q^{-\ell(e^{\lambda})/2}\t_{e^\lambda}\textrm{
if $\lambda\in \Cute^+(F)$}$.\\ 
\item[ii.] There is a unique
morphism of  $\Z[q^{\pm 1/2}]$-algebras
\beq\Theta_{F}^-:\Z[q^{\pm 1/2}][\tilde \X_*(\T)] \longrightarrow  \Htg\otimes_\Z \Z[q^{\pm 1/2}]\eeq
such that $\Theta^-_{F}(\lambda)=q^{-\ell(e^{\lambda})/2}\t_{e^\lambda}\textrm{ if $\lambda\in \Cute^-(F)$.}$\\

\item[iii.] Both $\Theta_{F}^+$ and $\Theta_{F}^-$ are injective.
\end{itemize}

\end{pro}

\begin{proof}[Proof of the proposition]  
It is the same proof as in the classical case  for Iwahori-Hecke algebras  and the dominant Weyl chamber.
Let $\sigma$ be a sign.
By Remark   \ref{rema:length}, the formula  $\Theta^\sigma_{F}(\lambda)=q^{-\ell(e^{\lambda})/2}\t_{e^\lambda}$  for 
$\lambda\in \Cute^\sigma(F)$ defines a multiplicative map $\Z[q^{\pm 1/2}][\Cute^\sigma(F)] \longrightarrow  \Htg\otimes_\Z \Z[q^{\pm 1/2}]$. Let  $\nu\in \Cute^\sigma(F)$ such that  $\lambda+ \nu\in \Cute^\sigma(F)$(if $\sigma=+$, choose for example $\nu$ to be the $w_F$-conjugate of a suitable strongly dominant coweight). We set
$\Theta^\sigma_F(\lambda):= q^{(-\ell(e^{\lambda+ \nu})+\ell(e^{\nu}))/2}
\tau_{\lambda+\nu}\tau_{\nu}^{-1}.$
This formula does not depend on the choice on $\nu$ such that $\lambda+ \nu \in \Cute^\sigma(F)$
as can be seen by applying again Remark \ref{rema:length}. 

To check the injectivity, bear on Lemma \ref{fond} that states that for all $\lambda\in \tilde \X_*(\T)$, the element $\Theta_F^\sigma(\lambda)$ is equal to the sum $q^{-\ell(\lambda)/2}\t_{e^\lambda}$  and of  a $\Z[q^{\pm 1/2}]$-linear combination of elements $\t_v\in \Htg$ such that $v\in \tilde \W$ and $\ell(v)<\ell(e^\lambda)$.

\end{proof}

\subsubsection{Commutative subalgebras  of $\Htg$}
For all $ \lambda\in\tilde \X_*(\T)$, we set
\begin{equation} \Bb_F^+(\lambda)= q^{\ell(e^{\lambda})/2}\Theta_F^+(\lambda) \:\:  \textrm{and}\:\:\Bb_F^-(\lambda)= q^{\ell(e^{\lambda})/2}\Theta_F^-(\lambda)\end{equation} and we remark in particular that by Lemma \ref{fond},  all these elements lie in $ \Htg$. 
\begin{rem}i. The maps $\Bb_F^+$ and $\Bb_F^-$ do not respect the product in general, but they are mulitplicative within Weyl chambers (see Remark \ref{rema:length}).\\
\label{same} 
ii. Consider the case $F=C$  or $F= x_0$. Then  $\Bb_{C}^+=\Bb_{x_0}^-$ (resp.  $\Bb_{C}^-=\Bb_{x_0}^+$) coincides with the integral Bernstein map $E^+$ (resp. $E$) introduced in \cite{Ann}. \\
 
\end{rem}

\begin{lemma}
For 
   $ \lambda\in \tilde \X_*(\T)$ we have
  $\upiota(\Bb_F^+(\lambda))=(-1)^{\ell(e^\lambda)} \Bb_F^-(\lambda) $ where $\upiota$ is the involution defined in \eqref{involution}.

\label{lemmainvo}
 
\end{lemma}

\begin{proof} Let $  \lambda\in\tilde \X_*(\T)$ and $\mu, \nu \in \Cute^+(F)$
 such that $\lambda= \mu -\nu$. Then
 $\Bb_F^+(\lambda)=q^{(\ell(e^{\lambda})+ \ell(e^{\nu})- \ell(e^{\mu}))/2} \tau_{e^{\mu}}  \tau_{e^{\nu}}^{-1}$ and
$ \Bb_F^-(\lambda)=q^{(\ell(e^{\lambda})+ \ell(e^{\mu})- \ell(e^{\nu}))/2} \tau_{e^{-\nu}}  \tau_{e^{-\mu}}^{-1}.$
Furthermore, $\ell(e^{\mu})- \ell(e^{\nu})$ and $\ell(e^\lambda)$ have the same parity  and
$\t_{e^{-\nu}} $ and $ \t_{e^{-\mu}}^{-1}$ commute by Remark \ref{rema:length}.  
\end{proof}

Using Lemma \ref{fond} we get the following:

\begin{pro}  Let $F$ be a standard facet. The commutative ring $$\Aa_F^+:=  \Htg\cap \Im(\Theta_F^+), \textrm{ and respectively }\Aa_F^-:=  \Htg\cap \Im(\Theta_F^-), $$ has $\Z$-basis the set of all $\Bb_F^+(\lambda)$, respectively  $\Bb_F^-(\lambda)$, for $\lambda\in \tilde \X_*(\T)$.\end{pro}

\begin{pro}\label{prop:commu}
Let $\lambda\in  \tilde \X^+_*(\T)$. \\$\bullet$ For any $t\in \T^0/\T^1$, the basis
element $\t_t\in \Htg$ and  $\Bb_F^+(\lambda)$ commute, as well as  $\t_t$ and  $\Bb_F^-(\lambda)$. \\ \\$\bullet$  Let $\root\in \Pi$.   If $\root\in\Pi_F$,  then

\medskip
 $a) \left\lbrace\begin{array}{ll}
\Bb_F^-(\lambda) \t_{n_\root}^*\in  q \Htg& \textrm{ if $\lp\lambda, \alpha\rp>0$}\cr
 \Bb_F^-(\lambda) \t_{n_\root}^*\in  \t_{n_\root}^*\Htg& \textrm{  if $\lp\lambda, \alpha\rp=0$}\cr\end{array}\right.$
 and $a')\:\left\lbrace\begin{array}{ll}
\Bb_F^+(\lambda) \t_{n_\root}\in  q \Htg& \textrm{ if $\lp\lambda, \alpha\rp>0$}\cr
 \Bb_F^+(\lambda) \t_{n_\root}\in  \t_{n_\root}\Htg& \textrm{  if $\lp\lambda, \alpha\rp=0$}\cr\end{array}\right.$\\\\
 
%\end{itemize}

$\bullet$  If $\root\in\Pi-\Pi_F$, then
 
\medskip
%\begin{itemize}
$b) \left\lbrace\begin{array}{ll}
\Bb_F^-(\lambda) \t_{n_\root}\in  q \Htg& \textrm{ if $\lp\lambda, \alpha\rp>0$}\cr
 \Bb_F^-(\lambda) \t_{n_\root}\in  \t_{n_\root}\Htg& \textrm{  if $\lp\lambda, \alpha\rp=0$}\cr\end{array}\right.$
and
$b') \left\lbrace\begin{array}{ll}
\Bb_F^+(\lambda) \t_{n_\root}^*\in  q \Htg& \textrm{ if $\lp\lambda, \alpha\rp>0$}\cr
 \Bb_F^+(\lambda) \t_{n_\root}^*\in  \t_{n_\root}^*\Htg& \textrm{  if $\lp\lambda, \alpha\rp=0$}\cr\end{array}\right.$

%\end{itemize}

\end{pro}

\begin{proof} 
Let $\nu\in \tilde\X(\T)$ be an element whose image in $\X(\T)$ is the opposite of  a strongly $F$-positive element (\S\ref{subsec:strongly}) and such that $\lambda+\nu \in \Cute_F^-(F)$.
Remark that   $\nu \in \Cute_F^-(F)$  and $e^\nu$    is an element in $\tilde \W$ that commutes with all 
$n_\root$, $\root\in \Pi_F$. We have $$\Bb_F^-(\lambda)= q^{\frac{\ell(e^\lambda)}{2}}\Theta_F^-(\lambda+\nu)\Theta_F^-(-\nu)= q^{\frac{\ell(e^\lambda)+\ell(e^{\nu})-\ell(e^{\lambda+\nu})}{2}} \t_{{e^{\lambda+\nu}}} 
 \t_{{e^{\nu}}}^{-1}.$$

\begin{itemize}
\item[a)] Let $\root\in\Pi_F$. Recall that $\t_{n_\root}^*=  q \t_{n_\root}^{-1}$. Since $s_\root$ and ${e^{\nu}}$ commute in $\W$, we have   $\ell(s _\root e^{\nu})=
\ell(e^{\nu})+1$ and $\t_{n_\root}$ and $\t_{e^\nu}$ commute: $$ \Bb_F^-(\lambda)\t_{n_\root}^*=  q^{\frac{2+\ell(e^\lambda)+\ell(e^\nu)-\ell(e^{\lambda+\nu})}{2}}\t_{e^{\lambda+\nu }}\t_{n_\root}^{-1}\t_{e^{\nu}}^{-1}$$

\begin{itemize}
 \item First suppose that $\lp\lambda, \alpha\rp>0$. Then $\ell(e^{\lambda} s_\root)=\ell(e^{\lambda})-1$
and  $\ell(e^{\lambda+\nu} s_\root)=\ell(e^{\lambda+\nu})-1$. Therefore,
$\t_{e^{\lambda+\nu}}=\t_{e^{\lambda+\nu}n_\root^{-1}}\t_{n_\root}$ and
$$ \Bb_F^-(\lambda)\t_{n_\root}^*=  q^{\frac{2+\ell(e^\lambda)+\ell(e^\nu)-\ell(e^{\lambda+\nu})}{2}}\t_{e^{\lambda+\nu }n_\root^{-1}}\t_{e^{\nu}}^{-1}=  q^{\frac{2+\ell(e^\lambda n_\root^{-1})+\ell(e^\nu)-\ell(e^{\lambda+\nu} n_\root^{-1})}{2}}\t_{e^{\lambda+\nu }n_\root^{-1}}\t_{e^{\nu}}^{-1}$$
which is an element  of  $q\Htg $ by Lemma \ref{fond}.\\

\item  Now suppose that $\lp\lambda, \alpha\rp=0$ so that $e^\lambda$, $e^\nu$ and $\t_{n_\alpha}$ commute.
We have $\ell(s_\root e^{\lambda+\nu})=\ell(e^{\lambda+\nu})+1$ so $ \t_{e^{\lambda+\nu}}\t_{n_\root}^{-1}
=\t_{n_\root}^{-1}\tau_{e^{\lambda+\nu}}$

$$ \Bb_F^-(\lambda)\t_{n_\root}^*=  q^{\frac{2+\ell(e^\lambda)+\ell(e^\nu)-\ell(e^{\lambda+\nu})}{2}}\t_{n_\root}^{-1}\t_{e^{\lambda+\nu }}\t_{e^{\nu}}^{-1}=\t_{n_\root}^* \Bb_F^-(\lambda).$$

\end{itemize}
\item[b)] Let $\root\in \Pi-\Pi_F$. We have  $\lp \nu, \root\rp<0$ so that $\ell(e^{\nu} s_\root)= \ell(e^{\nu})+1$
and $\t_{e^{\nu}}^{-1}\t_{n_\root}=\t_{n_\root} \t_{e^{s_\root \nu s_\root}}^{-1}$:
$$\Bb_F^-(\lambda)\t_{n_\root}= q^{\frac{\ell(e^\lambda)+\ell(e^{\nu})-\ell(e^{\lambda+\nu})}{2}} \t_{e^{\lambda+\nu}}\t_{n_\root}
 \t_{e^{s_\root\nu s_\root}}^{-1}.$$
Since  $\lp \nu+\lambda, \root\rp\leq0$ we have $\ell(e^{\nu+\lambda} s_\root)= \ell(e^{\nu+\lambda})+1$ and
$$\Bb_F^-(\lambda)\t_{n_\root}= q^{\frac{\ell(e^\lambda)+\ell(e^{\nu})-\ell(e^{\lambda+\nu})}{2}} \t_{e^{\lambda+\nu}n_\root}
 \t_{e^{s_\root\nu s_\root}}^{-1}.$$

\begin{itemize}
\item First suppose that $\lp\lambda, \alpha\rp>0$. Then $\ell(e^{\lambda} s_\root)=\ell(e^{\lambda})-1$
$$\Bb_F^-(\lambda)\t_{n_\root}= q^{\frac{2+ \ell(e^\lambda n_\root)+\ell(e^{s_\root\nu s_\root})-\ell(e^{\lambda+\nu}n_\root)}{2}} \t_{e^{\lambda+\nu}n_\root}
 \t_{e^{s_\root\nu s_\root}}^{-1}$$ which is an element  of  $q\Htg $ by Lemma \ref{fond}.\\

\item Now suppose that $\lp\lambda, \alpha\rp=0$ that it to say that $e^\lambda$ and $s_\root$ commute.
We have  $\t_{e^{\lambda+\nu}}\t_{n_\root}=\t_{n_\root} \t_{e^{\lambda+s_\root \nu s_\root}}$ and
$$\Bb_F^-(\lambda)\t_{n_\root}= \t_{n_\root}q^{\frac{\ell(e^\lambda)+\ell(e^{\nu})-\ell(e^{\lambda+\nu})}{2}} \t_{e^{\lambda+ s_\root\nu s_\root}}
 \t_{e^{s_\root\nu s_\root}}^{-1}.$$ which is an element  of  $\t_{n_\root}\Htg $ by Lemma \ref{fond}.\\

  \end{itemize}

\end{itemize}
Statements a') and b') follow applying Lemma \ref{lemmainvo} since 
$\upiota(\t_{n_\root})=\t_{n_\root}^*\t_{h_\root(-1)}$ and 
$\t_{h_\root(-1)}$  is invertible in
$\Htg$.
\end{proof}

\medskip

\subsubsection{\label{Levi}} Let $\M_F$ be the Levi subgroup of $\Gp$ corresponding to the facet $F$ as in \ref{LeviDefi}. We also refer to the notations introduced in \ref{HeckeRingF}.

\begin{lemma} \label{lemma:induction} For $\lambda\in \tilde \X_*^+(\T)$,  the element $(-1)^{\ell_F(e^\lambda)}\upiota^F(\t^F_{e^{\lambda}})\in \Hh_F(\M_F)$ is equal to the sum of $\t_{e^\lambda}^F$ and  a linear combination with coefficients in $\Z$  of   $\t_{\tilde v}^F$ for   $F$-positive elements $v\in \W_F$ such that $v\underset{F}< e^\lambda$. Furthermore, we have
\begin{equation}\label{f:induction}j_F^+( (-1)^{\ell_F(e^\lambda)}\upiota^F(\t^F_{e^{\lambda}})))= \Bb_F^+(\lambda).\end{equation}In particular  for $F= x_0$, 
\begin{equation}\label{f:inductionpart}(-1)^{\ell(e^\lambda)}\upiota(\t_{e^{\lambda}})= \Bb_{x_0}^+(\lambda).\end{equation}

\end{lemma}

\begin{proof}  In  $\Hh_\Z(\M_F)\otimes _{\Z}\Z[q^{\pm 1/2}]$, we have $(-1)^{\ell_F(e^{\lambda})}\upiota^F(\t
^F_{e^{\lambda}})= q^{\ell_F(e^\lambda)}(\t^F_{e^{-\lambda}})^{-1}$.  Lemma \ref{fond} for the Hecke algebra of $\M_F$  then gives the first statement. Use  Lemma \ref{F-positive} for the result about $F$-positivity.

\medskip

By an argument similar to the one in  the proof of Proposition \ref{themaps} (in the setting of the root system corresponding to $\M_F$), the element  $$\theta_F(\lambda):= q^{(\ell_F(\nu)-\ell_F(\mu))/2}\t^F_{e^{\mu}} (\t^F_{e^{\nu}})^{-1}\in \Hh_\Z(\M_F)\otimes _{\Z}\Z[q^{\pm 1/2}]$$ does not depend on the choice of $\lambda,\nu\in \X_*(\T)$ such that
$\lambda=\mu-\nu$ and $\lp \mu, \root\rp\leq 0$,  $\lp \nu, \root\rp\leq 0$ for all $\root\in \Phi_F^+$.

\noindent Choose  $\mu,\nu\in \Cute^+(F)$ such that $\lambda= \mu-\nu$. Then 
$j_F(q^{\ell_F(\lambda)/2}\theta_F(\lambda))= q^{(\ell_F(\lambda)+\ell_F(\nu)-\ell_F(\mu))/2}\t_{e^{\mu}} (\t_{e^{\nu}})^{-1}$ because $\mu$ and $\nu$ are in particular $F$-positive. By Equality \eqref{lengthF}, we therefore have 
$\Bb_F^+(\lambda)=j_F(q^{\ell_F(\lambda)/2} \theta_F(\lambda)) $. 
Now choose $\mu=0$ and $\nu=-\lambda$. We have 
$q^{\ell_F(\lambda)/2} \theta_F(\lambda)=(-1)^{\ell_F(e^{\lambda})}\upiota^F(\t
^F_{e^{\lambda}})$ and therefore $j_F(q^{\ell_F(\lambda)/2} \theta_F(\lambda)) =j_F^+((-1)^{\ell_F(e^{\lambda})}\upiota^F(\t
^F_{e^{\lambda}}))$.

\end{proof}

\medskip

\subsection{\label{subsec:satake}Satake isomorphism}
Let $\chi$ be a character of $\H$ with values in $k$ and $F_\chi$ the associated standard facet
as in \ref{param}. \begin{lemma} We have a morphism of $k$-algebras
\begin{equation}\label{firstiso}\begin{array}{ccc}
\bar\chi\otimes_{k[\T^0/\T^1]}k[ \tilde \X^+_*(\T)]&\longrightarrow &\Hom_\Hh(\chi\otimes_\H\Hh,\chi\otimes_\H\Hh )\cr
1\otimes\lambda&\longmapsto & (1\otimes 1\mapsto 1\otimes \Bb_{F_\chi}^+(\lambda))\cr\end{array}
\end{equation}

\end{lemma}

\begin{proof} We have to check  that, for $\lambda\in \tilde \X^+_*(\T)$, the element $1\otimes \Bb_{F_\chi}^+(\lambda)$ is an eigenvector for the right action of $\H$ and the character $\chi$. Recall that the finite Hecke algebra $\H$ is generated by all $\t_t$, $t\in \T^0/\T^1$ and $\t_{n_\root}$ for $\root\in \Pi$.\begin{itemize}\item  First note that for  $t\in \T^0/\T^1$, we have $\Bb_F^+(t+\lambda)=\t_t\Bb_F^+(\lambda)$. Therefore $\t_t$ acts on   $1\otimes \Bb_{F_\chi}^+(\lambda)$  by multiplication by $\chi(\t_t)$ and $\epsilon_{\bar \chi}$ acts by $1$.
\item  Let $\root\not\in \Pi_{\bar\chi}$. We have $\chi(\t_{n_\root})=0$. By the quadratic relations \eqref{Q'}, we have $\epsilon_{\bar\chi}\t_{n_\root}^*=
\bar\chi(h_\root(-1))\,\epsilon_{\bar\chi}\t_{n_\root}$ in $\Hh$. Since $ \Pi_\chi\subseteq \Pi_{\bar\chi}$,
 proposition \ref{prop:commu} b') implies that $\t_{n_\root}$ acts by $0$ on $1\otimes \Bb_{F_\chi}^+(\lambda)$.\item Let $\root\in \Pi_{\bar\chi}- \Pi_{\chi}$. 
We have $\chi(\t_{n_\root})=-1$ and by the quadratic relations \ref{Q'}, 
$\epsilon_{\bar\chi}\t_{n_\root}^*=
\epsilon_{\bar\chi}(\t_{n_\root}+1)$ in $\Hh$, which by proposition \ref{prop:commu} b'),  acts by $0$ 
on $1\otimes \Bb_{F_\chi}^+(\lambda)$.
\item  Let $\root\in  \Pi_{\chi}$. 
We have $\chi(\t_{n_\root})=0$ and by proposition \ref{prop:commu} a'),  $\t_{n_\root}$ acts by $0$
on $1\otimes \Bb_{F_\chi}^+(\lambda)$.\end{itemize}
We have proved that   \eqref{firstiso} is a well defined map.
It is  a morphism of $k$-algebras by Remark \ref{same}i.

\end{proof}
\begin{pro}\label{prop:main} The map \eqref{firstiso} is an isomorphism of $k$-algebras.

\end{pro}

\begin{proof} The proof relies on the following observation:
a basis for $\chi\otimes_\H \Hh$ is given by  all $1\otimes \t_{\tilde d }$ for $d\in \Dd$ (Proposition \ref{libertepro}). Recall that $\Dd$ contains the set of all $e^\mu$ for $\mu\in  \X^+_*(\T)$. By Lemma \ref{fond} (and using the braid relations \eqref{braid} together with \eqref{addi}), 
$1\otimes \Bb_{F_\chi}^+(\tilde\mu)$ is a sum of $\t_{\widetilde{e^\mu}}$ and of elements in  $\oplus _{d<e^\mu}k\otimes \t_{\tilde d}$.     

We first deduce  from this the injectivity of \eqref{firstiso}  because  a 
basis for $\bar\chi\otimes_{k[\T^0/\T^1]}k[ \tilde \X^+_*(\T)]$  is given by the set of all    $1\otimes \widetilde{e^\mu}$ for $\mu\in \X_*^+(\T)$.
\medskip

Now we  prove the surjectivity.  Denote,   for  $\mu\in \X^+_*(\T)$, by $\Hh[\mu]$ the subspace of  the functions in $\Hh$ with support in $\K \widehat {e^\mu}\K$. 
Then $\Hom_{\H}(\chi, \chi\otimes _\H\Hh)$
 decomposes into the direct sum of all subspaces $\Hom_{\H}(\chi, \chi\otimes _\H\Hh[\mu])$ for $\mu\in\X_*^+(\T)$ and after Corollary \ref{CoroIso} and its proof, each of the spaces $\Hom_{\H}(\chi, \chi\otimes _\H\Hh[\mu])$ is at most one dimensional.

Let $\mu \in \X_*^+(\T)$. By Lemma \ref{photo}ii and the observation at the beginning of this proof,  the image of $1\otimes \Bb_{F_\chi}(\tilde\mu)$
by \eqref{firstiso}  decomposes in  the direct sum of  all  $\Hom_{\H}(\chi, \chi\otimes _\H\Hh[\lambda])$  for $e^\lambda\leq e^\mu$ and it has a non zero component in $\Hom_{\H}(\chi, \chi\otimes _\H\Hh[\mu])$.
We conclude  by induction on $\ell(e^\mu)$ that $\Hom_{\H}(\chi, \chi\otimes _\H\Hh[\mu])$ is contained in the image of \eqref{firstiso} for all $\mu\in \X_*^+(\T)$.
\end{proof}

\bigskip

\begin{rem}
Recall that given  $\lambda\in \X^+_*(\T)$, 
a lift  for $e^\lambda\in\W$ is given by $\lambda(\varpi^{-1})\in \T$  (see \ref{cartan}).  More precisely, the map $\lambda\mapsto \lambda(\varpi^{-1})\,\rm{mod}\, \T^1$ is a splitting for the exact sequence of  abelian groups 
$$0\longrightarrow \T^0/\T^1\longrightarrow \tilde\X_*(\T) \longrightarrow \X_*(\T)\longrightarrow 0$$ and it respects the actions of $\Wf$. 

\end{rem}

By  abuse of notation, we identify below the element $\lambda(\varpi^{-1})\in N_\Gp(\T)$  with image in $\tilde \X^+_*(\T)\subset \tilde\W$.

\bigskip

Let $(\rho, \V)$ be the weight corresponding to the character $\chi$ of $\H$.  As in \ref{subsec:bcn}, we fix
a basis $v$ for $\rho^\I$. 
Composing
\eqref{firstiso} with  the inverse of \eqref{bcn} gives the following.

\begin{theorem}\label{theomy}We have an isomorphism
\begin{equation}\label{mysatake}\bar\chi\otimes_{k[\T^0/\T^1]}k[ \tilde \X^+_*(\T)]\overset{\sim}\longrightarrow \cal H(\Gp, \rho)\end{equation} 
carrying, for  $\lambda\in \X_*^+(\T)$, the element $1\otimes \lambda(\varpi^{-1})$ onto  the $\Gp$-equivariant map determined by
\begin{equation}\begin{array}{cccc}\EuScript T_\lambda:&\ind_\K ^\Gp\rho&\longrightarrow &\ind_\K ^\Gp\rho\cr& \1_{\K, v}&\longmapsto & \1_{\K, v} \Bb_{F_\chi}(\lambda(\varpi^{-1})).\cr
\end{array}\label{Tl}\end{equation}
\end{theorem}

\bigskip

\medskip

The map
$ \lambda\in\X_*^+(\T)\rightarrow \lambda(\varpi^{-1}) \,{\rm mod}\,\T^1$
yields an isomorphism
$$k[\X_*^+(\T)]\simeq  \bar\chi\otimes_{k[\T^0/\T^1]}k[ \tilde \X^+_*(\T)]$$  which we compose with  \eqref{mysatake} to obtain 
the isomorphism of $k$-algebras
\begin{equation}\label{invsatake}\begin{array}{cccc}\EuScript T:&k[ \X^+_*(\T)]&\overset{\sim}\longrightarrow& \cal H(\Gp, \rho)\cr \: &\lambda&\longmapsto &\EuScript{T}_{\lambda}\cr\end{array} \end{equation}

The next section  is devoted to  proving that,   in the case where the derived subgroup of $\mathbf G$ is simply connected,
this map is an inverse to the Satake isomorphism constructed  in \cite{satake}.
\section{\label{expli}Explicit computation of the mod $p$ modified Bernstein maps}
\subsection{Support of the  modified Bernstein functions} 
\subsubsection{Preliminary lemmas}

 \begin{lemma}\label{F-pos2prime} 
Let  ${\bf 1}: \T^0/\T^1\rightarrow k^\times$ be the trivial character of $\T^0/\T^1$ and $\epsilon_{\bf 1}\in \Hh$ the corresponding idempotent. For any $w\in\W$, 
we have in $\Hh$ the following equality:
\begin{equation}\label{calc}(-1)^{\ell (w)}\upiota (\epsilon_{\bf 1}\t  _{\tilde w})=\sum_{v\in \W , v   \leq w} \epsilon_{\bf 1}\t _{\tilde v}.\end{equation}

\end{lemma}

\medskip

\begin{proof}  We consider in this proof  the field $k$ as the residue field of an algebraic closure $\overline {\mathbb Q}_p$ of the field of $p$-adic numbers $\mathbb Q_p$. Let $\overline \Z_p$ be the ring of integers of  $\overline {\mathbb Q}_p$ and $r: \overline \Z_p\rightarrow k$ the reduction.
The ring ${\overline {\mathbb Z}_p}$ satisfies the hypotheses of \ref{subsec:fourier}.  In this proof we identify $q$ with  its image $q.1_{\overline {\mathbb Z}_p}$ in ${\overline {\mathbb Z}_p}$.
We work in the Hecke algebra $\Hh_\Z\otimes _\Z {\overline {\mathbb Z}_p}$ in which we prove   that \begin{equation}\label{tbr}(-1)^{\ell (w)}\upiota (\epsilon_{\bf 1}\t  _{\tilde w}) \: \in \: \sum_{v\in \W , v   \leq w} (1-q)^{\ell (w)-\ell (v)} \epsilon_{\bf 1}\t _{\tilde v} + q \:( \Htg\otimes_\Z {\overline {\mathbb Z}_p}).\end{equation} 

 It is enough to consider the case  $w\in \W_{aff}$ and
we proceed by induction on $\ell (w)$.
Let $w\in \W_{aff} $ and $s\in \SS_{aff} $  such that with $\ell(sw)=\ell(w)+1$.  Applying \cite[Lemma 4.3]{Haines}, we see that
the set of the $v\in \W $ such that $v  \leq sw$ is the disjoint union of 
$$\{v\in \W , v  \leq sv , w\}\:\textrm{and} \: \{v\in \W , sv  \leq v, w  \}. $$
Noticing that ${\epsilon_1}\upiota(\t_{\tilde s})=-\epsilon_{\bf 1}(\t_{\tilde s}+1-q)$, we have, by induction,
$$(-1)^{\ell (\widetilde{sw})}\epsilon_{\bf 1}\upiota (\t   _{\widetilde{sw}})= (-1) ^{\ell (w)}\epsilon_{\bf 1}(\t_{\tilde s }   +1-q)\upiota (\t   _{\tilde w})\in \epsilon_{\bf 1}(\t_{\tilde s}   +1-q)\sum_{v\in \W , v   \leq w} (1-q)^{\ell (w)-\ell (v)} \t   _{\tilde v} + q(\Htg\otimes_\Z {\overline {\mathbb Z}_p})$$ and $(\t_{\tilde s}   +1-q)\underset{v\in \tilde\W , v   \leq w}\sum (1-q)^{\ell (w)-\ell (v)} \epsilon_{\bf 1}\t   _{\tilde v}$ is successively equal to

\begin{align*}& \epsilon_{\bf 1}(\t_{\tilde s}   +1-q)(\sum_{{v   \leq sv, w}} (1-q)^{\ell (w)-\ell (v)} \tg   _{\tilde v}+ \sum_{ {sv  \leq v   \leq w}} (1-q)^{\ell (w)-\ell (v)} \tg   _{\tilde v}) \cr&=\sum_{ v   \leq sv,w} (1-q)^{\ell (w)-\ell (v)} \tg   _{\widetilde{sv}}+ 
\sum_{ v   \leq sv,w} (1-q)^{\ell (w)-\ell (v)+1} \tg   _{\tilde{v}}+\sum_{ {sv  \leq v   \leq w}} q(1-q)^{\ell (w)-\ell (v)} \tg   _{\widetilde{sv}} \cr
&\in \sum_{sv   \leq v,w}(1-q)^{\ell (sw)-\ell (v)} \tg   _{\tilde{v}}+ 
\sum_{v   \leq sv,w} (1-q)^{\ell (sw)-\ell (v)} \tg   _{\tilde{v}}+\sum_{ {sv  \leq v   \leq w}} q(1-q)^{\ell (w)-\ell (v)} \tg   _{\widetilde{sv}}\cr
\end{align*}
which proves the  claim. Applying the reduction   $r: \overline \Z_p\rightarrow k$, we get \eqref{calc} in $\Hh$.
 
\end{proof}

\begin{lemma}\label{extend} Suppose that the derived subgroup of $\mathbf G$ is simply connected.\\
Let $\xi:\T^0/\T^1\rightarrow k^\times$ be  a character fixed by all $s\in \SS$. Then there exists a  character $\upalpha:\Gp\rightarrow k^\times$ that coincides with $\xi^{-1}$ on $\T^0$, such that
$\upalpha(\mu(\varpi))=1$ for all $\mu\in \X_*(\T)$. It satisfies the following equality in $\Hh$, for $\lambda\in \X_*^+(\T)$:
$$\epsilon_\xi\Bb^+_{x_0}(\lambda(\varpi^{-1}))= \epsilon_{\xi}(-1)^{\ell (e^{\lambda})}\upiota (\t  _{\lambda(\varpi^{-1})})=  \sum_{v\in \W , v   \leq e^\lambda} \epsilon_{\xi}\upalpha( {\tilde v})\t _{\tilde v}.$$
% (-1)^{\ell (e^\lambda)}\upiota (\epsilon_{\xi}\t  _{\lambda(\varpi^{-1})})

\end{lemma}

\begin{rem} \label{alphaindep}For $v\in  \W$ with chosen lift $\tilde v\in \tilde\W$,
the value of $\upalpha( \hat{\tilde v})$ does not depend on the choice of $\hat{\tilde v}$ lifting $\tilde v$ and
we denote it by $\upalpha(\tilde v)$ above.
Furthermore, $\epsilon_{\xi}\upalpha( {\tilde v})\t _{\tilde v}$  does not depend on the choice of the lift $\tilde v$.

\end{rem}

\begin{proof} 
Define a character  $\upchi:\H\rightarrow k^\times$  by   $\bar\upchi:=\xi$ and $\Pi_\upchi=\Pi_{\bar\upchi}= \Pi$ and  consider the associated weight. By the proof of \cite[Proposition 5.1]{Parabind} (see also the remark following Definition 2.4 \emph{loc.cit} and Remark \ref{stabi}),
this weight   is a character $\K\rightarrow k^\times$ and by \cite[Corollary 3.4]{abe}, it extends uniquely to a character $\upalpha$ of $\Gp$ satisfying $\upalpha(\mu(\varpi))=1$ for all $\mu\in \X_*(\T)$. Note that $\upalpha$ coincides with $\bar \upchi^{-1}$ on $\T^0/\T^1$.
\begin{factn}We have an isomorphism of $k$-algebras
$\Psi:\Hh\rightarrow \Hh, \: \t_g\mapsto  \upalpha(g)\:\t_g$  preserving the support of the functions. It sends $\epsilon_{\bf 1}$ onto $\epsilon_{\xi}$ and commutes with the involution $\upiota$

\end{factn}

\begin{proof}[Proof of the fact]
The image of $\t_g=\1_{\I g\I}$ is independent from the choice of a representative in $\I g\I$.
The image of $\epsilon_{\bf 1}$ is clearly $\epsilon_{\xi}$. One easily checks that $\Psi$ respects the product.
Now in order to check that $\Psi$ commutes with the involution $\upiota$, it is enough to show that
$\Psi$ and $\upiota \Psi \upiota$ coincide on elements of the form $\t_u$ with $u\in \tilde\W$ such that $\ell(u)=0$ and $u=n_A$ for $A\in \Pi_{aff}$. For the former elements, the claim is clear since $\upiota$ fixes such $\t_u$ when $\ell(u)=0$. Now let $A=(\alpha, r)\in \Pi_{aff}$.  We consider the morphism
$\phi_\root: {\rm SL}_2(\Corps)\rightarrow \mathbf G(\Corps)$ as in \ref{rootsubgroup}.
Since $\Corps$ is infinite, %${\rm SL}_2(\Corps)$ is equal to its own derived subgroup and 
the restriction of $\upalpha$ to the image of $\phi_\root$ is trivial.
Now by Remark \ref{invoNA},  we have
 $\upiota\Psi(\upiota(\t_{n_A}))=\t_{n_A}.$

\end{proof}

We deduce from this (and using \eqref{f:inductionpart}) that  $$\Psi(\epsilon_{\bf 1}\Bb^+_{x_0}(\lambda(\varpi^{-1})))=\epsilon_{\xi}(-1)^{\ell (e^{\lambda})}\Psi(\upiota (\t  _{\lambda(\varpi^{-1})}))=\epsilon_{\xi}\Bb^+_{x_0}(\lambda(\varpi^{-1})).$$ Conclude using Lemma \ref{F-pos2prime}.

\end{proof}

\subsubsection{} For  $\chi: \H\rightarrow k^\times$, we consider  the associated facet  $F_\chi$  as in \ref{param} and    $\M_\chi$ the  corresponding standard Levi subgroup as in \ref{Levi}.

\begin{pro}  Let $\chi: \H\rightarrow k^\times$.  There is a character  $\upalpha_{\chi}: \M_{F_\chi}\rightarrow k^\times$  such that\\
a)   $\upalpha_{\chi}$ coincides with $\bar \chi^{-1}$ on $\T^0$ and satisfies $\upalpha_{\chi}(\mu(\varpi))=1$ for all $\mu\in \X_*(\T)$,\\b) in $\chi\otimes_{\H}\Hh$ we have,  for
 $\lambda\in \X^+_*(\T)$,   \begin{equation} \label{almostthere}\epsilon_{\bar\chi}\otimes \Bb_{F_\chi}(\lambda(\varpi^{-1}))=\sum_{d\in \W_{F_\chi}\cap \Dd_{F_\chi},  d   \underset{F}\leq e^\lambda}  \upalpha_{\chi}({\tilde d}) \: \epsilon_{\bar\chi}\otimes \t _{\tilde d}  \end{equation}where $\tilde d$   denotes any lift in $\tilde \W_{F_\chi}$  for $d\in \W_{F_\chi}$. 
\end{pro}

\begin{proof}  In this proof we write $F$ instead of $F_\chi$. Let $\lambda\in\X_*^+(\T)$.
\begin{comment}

Suppose first that $F= x_0$ \green{that it to say ???? $\M_\chi=\Gp$}. The character $\chi$ is then fixed by all $s\in \SS$ and $\chi(\t_{\tilde w})=0$ for all $w\in \Wf$.  Then by Lemma \ref{extend}, and by definition of $\Dd$,   there exists a character $\upalpha_\chi:\Gp\rightarrow k^\times$ that coincides with $\bar\chi^{-1}$ on $\T^0$, such that,
$\upalpha_\chi(\mu(\varpi))=1$ for all $\mu\in \X_*(\T)$ and 
$$\epsilon_{\bar\chi}\otimes \Bb_{x_0}(\lambda(\varpi^{-1}))= \sum_{v\in \Dd , v   \leq e^\lambda} \upalpha_\chi( {\tilde d}) \epsilon_{\bar\chi} \otimes \t _{\tilde d}.$$
Now for the general case,  
\end{comment}
Consider the restriction to $\M_{F}\cap \K$ of the weight $\rho$  associated to $\chi$ 
 and the corresponding restriction of $\chi$ to the finite Hecke algebra $\H_{x_{F}}(\M_{F})$ (see  Remark \ref{ParabParah} for the definition of this subalgebra of $\Hh(\M_{F})$ attached to the maximal compact subgroup $\M_{F}\cap \K$ of $\M_{F}$).   It satisfies the hypotheses of Lemma \ref{extend}, where $\M_{F}$ and its attached root system play the role of $\Gp$.  Note that under our hypothesis for $\Gp$, the derived subgroup of $\M_{F}$ is equally simply connected.
Therefore, there exists a character $\upalpha_\chi:\M_{F}\rightarrow k^\times$ that coincides with $\bar\chi^{-1}$ on $\T^0$, such that
$\upalpha_\chi(\mu(\varpi))=1$ for all $\mu\in \X_*(\T)$ and satisfying the following equality in $\Hh(\M_{F})$ (see \eqref{f:inductionpart} applied to the Levi $\M_{F}$)
$$\epsilon_{\bar\chi}\: (-1)^{\ell_{F}(e^\lambda)}\upiota^{F}(\t^{F}_{\lambda(\varpi^{-1})})= \sum_{v\in \W_{F} , v  \underset{{F}} \leq e^\lambda} \epsilon_{\bar\chi}\upalpha_{\chi}( {\tilde v})\t ^{F}_{\tilde v}.$$ Now 
applying   \eqref{f:induction} to the facet $F$, we have in $\Hh$,
$$\epsilon_{\bar\chi} \Bb_{F}^+(\lambda(\varpi^{-1}))= \sum_{v\in \W_{F} , v  \underset{F} \leq e^\lambda} \epsilon_{\bar\chi}\upalpha_{\chi}( {\tilde v})\t _{\tilde v}.$$ 
Before projecting this relation onto $\chi\otimes _{\H}\Hh$, recall that $\chi(\t_{\tilde w})=0$ for all $w\in \Wf_{F}$.
By definition of $\Dd_{F}$ (see also Remark \ref{rema:DF}), we therefore have, in $\chi\otimes _{\H}\Hh$,
$$\epsilon_{\bar\chi}\otimes \Bb_{F}^+(\lambda(\varpi^{-1}))= \sum_{d\in \W_{F}\cap \Dd_{F} , \: v  \underset{{F}} \leq e^\lambda} \upalpha_{\chi}( {\tilde d})\: \epsilon_{\bar\chi}\otimes \t _{\tilde d}.$$

\end{proof}
\subsection{\label{invH}An inverse to the mod $p$ Satake transform of \cite{satake}}
Let $(\rho, \V)$ be a weight,  $\chi: \H\rightarrow k^\times$  the corresponding character  and  $F_\chi$ the facet defined as in \ref{param}. Consider the isomorphism
  \begin{equation}\label{MapHerzig}\EuScript S: \cal H(\Gp, \rho)\overset\sim{\longrightarrow}k[\X_*^+(\T)]\end{equation} 
 constructed in  \cite{satake}
 (see Remark \ref{rema:normalization}).    For $\lambda\in \X_*^+(\T)$, denote by $f_\lambda$ the  function in  $\cal H(\Gp, \rho)$  with support in $\K\lambda(\varpi^{-1})\K$ and value at $\lambda(\varpi^{-1})$ equal to the $k$-linear projection $\V\rightarrow \V$ defined by \cite[(2.8)]{satake}. Note  that this projection coincides with the identity on $\V^\I$ (step 3 of the proof of Theorem 1.2 \emph{loc.cit}). Any function in $\cal H(\Gp, \rho)$  with support in $\K\lambda(\varpi^{-1})\K$  is a $k$-multiple of $f_\lambda$. 
The element  $T_{f_\lambda}(\1_{\K, v})$ defined by $g\mapsto f_\lambda(g)v$ is the unique element in $ (\ind_{\K}^\Gp\rho)^\I$ with support in 
 $\K\lambda(\varpi^{-1})\K$ and value $v$ at $\lambda(\varpi^{-1})$  which is an eigenvector for the action of $\H$ and the character $\chi$
(See Corollary \ref{CoroIso} and Remark \ref{dim1}).

\medskip
Recall that the isomorphism $\EuScript T:k[\X_*^+(\T)] \overset\sim{\longrightarrow}\cal H(\Gp, \rho)$ was defined in \eqref{invsatake} and that both $\EuScript S$ and $\EuScript T$ are defined with no further condition on the derived subgroup of $\mathbf G$.

\begin{theorem} Suppose that the derived subgroup of $\mathbf G$ is  simply connected.\\
\label{invherzig0}
\begin{itemize}
\item[i.] For   $\lambda\in  \X^+_*(\T)$ we have
\begin{equation}\EuScript T_\lambda=\sum_{\mu\underset{F_\chi}\preceq \lambda} f_\mu.\label{invherzig}\end{equation} 
\item[ii.] 
 \label{coro:invherzig}
The map $\EuScript T$ is an inverse for  $\EuScript S$.
\end{itemize}

\end{theorem}

\begin{rem} In particular, the matrix coefficients of $\EuScript T$ in the bases $\{\lambda\}_{\lambda\in \X_*^+(\T)}$ and $\{f_\lambda\}_{\lambda\in \X_*^+(\T)}$    depend only on the facet $F_\chi$, and not on $\chi$ itself.

\end{rem}

\begin{proof}  In this proof, we write $F$ for $F_\chi$.
For i, we have to show that $\EuScript T_\lambda$ has support the  set of  all double cosets $\K \mu(\varpi^{-1})\K$ for  $\mu\underset{F}\preceq \lambda$ and that for  $v\in \rho^\I$ and  such a $\mu$, the value of $\EuScript T_\lambda(\1_{\K, v})$ at $ \mu(\varpi^{-1})$  is $v$.  By
 \eqref{almostthere},  we have
 $$\EuScript T_\lambda(\1_{\K, v})=  \sum_{d\in \W_F\cap \Dd_F,  d   \underset{F}\leq e^\lambda}  \upalpha_{\chi}(\tilde d)\1_{\K, v}\t _{\tilde d}$$  and by Lemma \ref{F-positive}ii, this element  has support in 
 the expected set. Furthermore, by the proof of Proposition \ref{apple}, it has value $\upalpha_{\chi}(\tilde d)v$ at $\hat {\tilde d} $ for all
 $d\in \W_F\cap \Dd_F$,   $d   \underset{F}\leq e^\lambda$.
 Now recall that any $\mu\in \X_*^+(\T)$ seen in $\W$  is an element in $ \Dd_F$ (Remark \ref{rema:D}), that $\mu(\varpi^{-1})$ is a lift in $\Gp$ for $\mu$ and that $\upalpha_\chi(\mu(\varpi^{-1}))=1$. It proves that 
 $\EuScript T_\lambda(\1_{\K, v})$ has value $v$ at $\mu(\varpi^{-1})$ for $\mu\underset{F}\preceq \lambda$: we obtain the formula \eqref{invherzig}.
 Finally,  let $\lambda\in  \X^+_*(\T)$.  Under the hypothesis that the derived subgroup of $\mathbf G$ is  simply connected,  $\sum_{\mu\underset{F}\preceq \lambda} \EuScript S( f_\mu)$ is equal to the element $\lambda$ seen in $k[\X_*^+(\T)]$  \cite[Proposition 5.1]{Parabind}. It proves ii.

\end{proof}

\bibliographystyle{smfplain}

%\theendnotes
\end{document}